\documentclass{amsart}
\usepackage{amsfonts,amscd,amssymb,amsmath,amsthm,mathdots,upgreek}
\usepackage{stmaryrd}
\SetSymbolFont{stmry}{bold}{U}{stmry}{m}{n}
\usepackage{graphicx,caption,subcaption,mathrsfs,appendix}
\usepackage{color}
\usepackage{enumerate}
\usepackage{bm}
\usepackage[colorlinks=true,linkcolor=blue]{hyperref}

\numberwithin{equation}{section}

\usepackage{xparse}
\definecolor{airforceblue}{rgb}{0.36, 0.54, 0.66}
\definecolor{amethyst}{rgb}{0.6, 0.4, 0.8}
\definecolor{applegreen}{rgb}{0.55, 0.71, 0.0}

\def\corA{}

\def\corAB{}

\definecolor{purple}{rgb}{0.9,0,0.8}

\newcommand{\abbr}[1]{{\sc\lowercase{#1}}}
\begin{document}

\newtheorem{theorem}{Theorem}
\newtheorem{proposition}[theorem]{Proposition}
\newtheorem{conjecture}[theorem]{Conjecture}
\newtheorem{corollary}[theorem]{Corollary}
\newtheorem{lemma}[theorem]{Lemma}
\theoremstyle{definition}
\newtheorem{dfn}{Definition}
\newtheorem{assumption}[theorem]{Assumption}
\newtheorem{claim}[theorem]{Claim}
\newtheorem{remark}[theorem]{Remark}

\numberwithin{theorem}{section}
\numberwithin{dfn}{section}

%%%%%%%%%%%%%%%%%%%%%%%%%%%%%%%%%%%%%%%%%%
%%% General macros
%%%%%%%%%%%%%%%%%%%%%%%%%%%%%%%%%%%%%%%%%%

\newcommand{\B}{\mathbb{B}}
\newcommand{\R}{\mathbb{R}}
\newcommand{\T}{\mathcal{T}}
\newcommand{\C}{\mathbb{C}}
\newcommand{\D}{\mathbb{D}}
\newcommand{\G}{\mathcal{G}}
\newcommand{\Z}{\mathbb{Z}}
\newcommand{\Q}{\mathbb{Q}}
\newcommand{\E}{\mathbb E}
\renewcommand{\P}{\mathbb P}
\newcommand{\N}{\mathbb N}
\newcommand{\gd}{\mathfrak{d}}
\newcommand{\gb}{\mathfrak{b}}
\newcommand{\gL}{\mathfrak{L}}
\newcommand{\vep}{\varepsilon}
\newcommand{\cS}{\mathcal{S}}
\newcommand{\cI}{\mathcal{I}}
\newcommand{\cY}{\mathcal{Y}}
\newcommand{\cB}{\mathcal{B}}
\newcommand{\cM}{\mathcal{M}}
\newcommand{\cN}{\mathcal{N}}
\newcommand{\cX}{\mathcal{X}}
\newcommand{\frakg}{\mathfrak{g}}
\newcommand{\supp}{{\mbox{\rm Supp }}}

\newcommand{\barray}{\begin{eqnarray*}}
\newcommand{\earray}{\end{eqnarray*}}

\newcommand{\dvec}[1]{ \llbracket #1 \rrbracket}

\newcommand{\Def}{:=}

%%%%%%%%%%%%%%%%%%%%%%%%%%%%%%%%%%%%%%%%%%
%%% Probability Macros
%%%%%%%%%%%%%%%%%%%%%%%%%%%%%%%%%%%%%%%%%%

\DeclareDocumentCommand \Pr { o }
{%
\IfNoValueTF {#1}
{\operatorname{Pr}  }
{\operatorname{Pr}\left[ {#1} \right] }%
}
\newcommand{\Prob}{\Pr}
\newcommand{\Exp}{\mathbb{E}}
\newcommand{\expect}{\mathbb{E}}
\newcommand{\1}{\one}
\newcommand{\Pto}{\overset{\mathbb{P}}{\to} }
\newcommand{\weakto}{\Rightarrow}
\newcommand{\lawequals}{\overset{\mathcal{L}}{=}}
\newcommand{\prob}{\Pr}
\newcommand{\pr}{\Pr}
\newcommand{\filt}{\mathscr{F}}
\newcommand{\ohadI}{\mathbbm{1}}
\DeclareDocumentCommand \one { o }
{%
\IfNoValueTF {#1}
{\ohadI }
{\ohadI\left\{ {#1} \right\} }%
}
\newcommand{\sgn}{\operatorname{sgn}}
\newcommand{\Bernoulli}{\operatorname{Bernoulli}}
\newcommand{\Binomial}{\operatorname{Binom}}
\newcommand{\Binom}{\Binomial}
\newcommand{\Poisson}{\operatorname{Poisson}}
\newcommand{\Exponential}{\operatorname{Exp}}

\newcommand{\Var}{\operatorname{Var}}
\newcommand{\Cov}{\operatorname{Cov}}

%%%%%%%%%%%%%%%%%%%%%%%%%%%%%%%%%%%%%%%%%%
%%% Random Matrix Macros
%%%%%%%%%%%%%%%%%%%%%%%%%%%%%%%%%%%%%%%%%%

\newcommand{\Id}{\operatorname{Id}}
\newcommand{\diag}{\operatorname{diag}}
\newcommand{\tr}{\operatorname{tr}}
\newcommand{\proj}{\operatorname{proj}}
\newcommand{\Span}{\operatorname{span}}

%%%%%%%%%%%%%%%%%%%%%%%%%%%%%%%%%%%%%%%%%%
%%% Paper-Specific Macros
%%%%%%%%%%%%%%%%%%%%%%%%%%%%%%%%%%%%%%%%%%

\newcommand{\nicefrac}{\frac}
\newcommand{\half}{\frac12}
\DeclareDocumentCommand \JB { O{n} O{\lambda} } {J_{{#1}}({#2})}

\DeclareDocumentCommand \LP { O{\ESD} } {U_{ {#1} }}
\newcommand{\LPL}{ \LP[{\mu_N}] }

\newcommand{\ESD}{ L_N^{\model} }

\newcommand{\model}{\mathcal{M}}
\newcommand{\SVD}{\Sigma}
\newcommand{\LL}{\mathcal{L}}
\newcommand{\PI}{\Pi}
\DeclareDocumentCommand \PG { O{n} }
{
\mathfrak{S}_{{ #1 }}
}
\newcommand{\RS}{\mathcal{C}}

\newcommand{\TODO}[1]{ {\bf TODO: #1} }

\newcommand{\row}{X}
\newcommand{\col}{Y}
\newcommand{\srow}{x}
\newcommand{\scol}{y}
\newcommand{\csrow}{w}
\newcommand{\cscol}{z}
\newcommand{\COMP}[1]{ \check{#1} }

\date{December 14, 2018. Revised November 7, 2019.}
\title[Spectrum of random perturbation of Toeplitz] {Spectrum of random perturbations of Toeplitz matrices with finite symbols}
%and twisted Toeplitz cases}
%\shortitle{Regularization of non-normal matrices}
\author[A.\ Basak]{Anirban Basak$^*$}%\thanks{${}^*$Partially supported by
 %grant 147/15 from the Israel Science Foundation, Start-up Research Grant (SRG/2019/001376) from Science and Engineering Research Board of Govt.~of India, and ICTS--Infosys Excellence Grant.}
 \address{$^*$International Center for Theoretical Sciences
\newline\indent Tata Institute of Fundamental Research
\newline\indent Bangalore 560089, India
\newline\indent and
\newline\indent Department of Mathematics, Weizmann Institute of Science
\newline\indent POB 26, Rehovot 76100, Israel}
 \author[E.\ Paquette]{Elliot Paquette$^\ddagger$}%\thanks{${}^\ddagger$Partially supported by NSF postdoctoral fellowship DMS-1606310}
 \address{$^\ddagger$Department of Mathematics, The Ohio State University  
 \newline\indent Tower 100, 231 W 18th Ave, Columbus, Ohio 43210, USA}
\author[O.\ Zeitouni]{Ofer Zeitouni$^{\mathsection}$}%\thanks{${}^{\mathsection}$Partially 
%supported by  grant 147/15 from the Israel Science Foundation}
\address{$^{\mathsection}$Department of Mathematics, Weizmann Institute of Science 
 \newline\indent POB 26, Rehovot 76100, Israel
 \newline \indent and
 \newline\indent Courant Institute, New York University
 \newline \indent 251 Mercer St, New York, NY 10012, USA}
%\author{Anirban Basak, Elliot Paquette and Ofer Zeitouni}
%  \address{Department of Mathematics, Weizmann Institute of Science}}
%  \author{Elliot Paquette
%    \address{Faculty fo Mathematics, Ohio State University}}
%    \author{Ofer Zeitouni
%\address{Department of Mathematics, Weizmann Institute of Science}
%\address{Faculty of Mathematics, Ohio State University}
  %\address{Department of Mathematics, Weizmann Institute of Science}

%\thanks{This work was partially supported by a grant from the Israel
%Science Foundation}

\begin{abstract}
  Let $T_N$ denote an $N\times N$ Toeplitz matrix with finite, $N$ independent
  symbol ${\bm a}$. For $E_N$ a noise matrix satisfying mild assumptions (ensuring, in particular, that $\corA{N^{-1/2}\|E_N\|_{{\rm HS}}}\to_{N\to\infty} 0$ at a polynomial rate), 
  we prove that the empirical measure
  of eigenvalues of $T_N+E_N$ converges to the law of ${\bm a}(U)$, where $U$
  is uniformly distributed on the unit circle in the complex plane. This extends results from \cite{BPZ} to the non-triangular setup and non complex Gaussian noise,
  and confirms
  predictions obtained in \cite{trefethen} using the notion of pseudospectrum.
\end{abstract}

%Ofer
\maketitle

%\tableofcontents

\section{Introduction}
Let $\mathbb{S}^1:=\{z \in \C: |z| =1\}$ denote
the unit circle in the complex plane. Let ${\bm a}: \mathbb{S}^1 \mapsto \C$ be a function given by
\[
{\bm a }(\lambda) := \sum_{k=-\infty}^\infty a_k \lambda^k, \qquad \lambda \in \mathbb{S}^1,
\]
where $\{a_k\}_{k=-\infty}^\infty$ is an absolutely summable complex valued sequence. 
We denote by 
$T_N:=T_N({\bm a})$  the Toeplitz matrix of dimension $N \times N$ with symbol ${\bm a}$, given by
\[
T_N:=\begin{bmatrix}
a_{0} & a_{1} & a_2 & \cdots & \cdots & a_{N-1}\\
a_{-1}& a_{0} & a_1 &\ddots &  & \vdots\\
a_{-2} &a_{-1}& \ddots & \ddots & \ddots & \vdots \\
\vdots & \ddots & \ddots & \ddots & a_1 & a_2\\
\vdots & & \ddots & a_{-1} & a_0 & a_1\\
a_{-(N-1)} & \cdots & \cdots & a_{-2} &a_{-1} & a_{0}
\end{bmatrix}.
\]
From the definition it is clear that when ${\bm a}$ is a Laurent polynomial, i.e.
\[
{\bm a}(\lambda)= \sum_{k=-d_2}^{d_1} a_k \lambda^k, \qquad \text{ for some } d_1, d_2 \ge 0,
\]
then
$T_N$ is a (finitely) banded Toeplitz matrix which can be thought of as
a piece
from an infinite Toeplitz matrix; we refer to such matrices as
%Such matrices will be called 
Toeplitz matrices with finite symbols. 

For any $N \times N$ matrix $A_N$ we denote the empirical measure 
of its eigenvalues, or equivalently \abbr{ESD}, the empirical spectral 
distribution, by $L_N^A$. That is, 
\[
L_N^A:= \frac{1}{N}\sum_{i=1}^n \delta_{\lambda_i},
\]
where $\lambda_1, \lambda_2, \ldots, \lambda_N$ are the eigenvalues of $A_N$.
In this paper, we find the limit of the empirical spectral distribution 
(\abbr{ESD}) of random perturbations of Toeplitz matrices with finite symbols.
This 
generalizes those results in 
\cite{BPZ} that deal with \textit{triangular} 
Toeplitz matrices with finite symbols (and also with
twisted Toeplitz matrices, which we cannot generalize to the 
non-triangular case, see
Remarks \ref{rem-twisted} and \ref{rem-ell} below). 
In contrast with \cite{BPZ}, we allow
for rather general perturbations, as codified in 
Assumption \ref{ass:noise-matrix}.
\begin{assumption}\label{ass:noise-matrix}
Let $\{E_N\}_{N \in \N}$ be a sequence of matrices, with possibly complex valued entries, such that the followings hold:
\begin{enumerate}
\item[(i)]
\[
\E \left[\sum_{i,j=1}^N |e_{i,j}|^2 \right] = O(N^2),
\]
where $\{e_{i,j}\}_{i,j=1}^N$ are the entries of $E_N$. 
\item[(ii)] For any $\alpha \in (0,\infty)$, there exists a $\beta \in (0,\infty)$, depending only on $\alpha$, so that for any fixed deterministic matrix $M_N$ with $\|M_N\| =O(N^\alpha)$, we have
\[
\P\left(s_{\min}(E_N+ M_N) \le N^{-\beta} \right) = o(1). 
\] 
\end{enumerate}
\end{assumption}

Let $\mu_{\bm a}$ denote the law of
${\bm a}(U)$, where $U$ is a random variable 
uniformly distributed on the unit circle in the complex plane. 
Equipped with Assumption \ref{ass:noise-matrix} we now state the main result of this paper. 
\begin{theorem}\label{thm:main}
Let $T_N$ be any $N \times N$ Toeplitz matrix with a symbol ${\bm a}$, where ${\bm a}$ is a Laurent polynomial. Assume that $\{E_N\}_{N \in \N}$ satisfy Assumption \ref{ass:noise-matrix}. Then, for any $\gamma >\frac12$,
the \abbr{ESD} $L_N^{T+N^{-\gamma}E}$
of $T_N+N^{-\gamma}E_N$ converges weakly, in probability, to $\mu_{\bm a}$. 
\end{theorem}
Assumption \ref{ass:noise-matrix}(i) holds as soon 
as the second moment of each of the 
entries (of both complex and real parts) is uniformly bounded. By
\cite[Theorem 2.1]{TV08}, 
whenever the entries of $E_N$ are i.i.d. (complex or real)
with common $N$-independent
distribution having a finite variance, 
Assumption \ref{ass:noise-matrix}(ii) holds. Therefore, Theorem 
\ref{thm:main} holds in that setup. In the next remark, we summarize other cases where Assumption \ref{ass:noise-matrix}, and hence Theorem \ref{thm:main}, hold. 
%  if 
%$E_N$ is a matrix whose entries are i.i.d. (complex or real) 
%and come from an
%$N$-independent fixed distribution $\mu$ possessing finite second moments.
%\begin{corollary}\label{cor:main}
%Let $T_N$ be any $N \times N$ Toeplitz matrix with a symbol ${\bm a}$, where ${\bm a}$ is a Laurent polynomial. Assume that the entries of $E_N$ are i.i.d.~with finite second moment. Then, for any $\gamma >\frac12$ the \abbr{ESD} of $T_N+N^{-\gamma}E_N$ converges weakly, in probability, to the law of ${\bm a}(U)$, where $U$ as in Theorem \ref{thm:main}. 
%\end{corollary}

%In the next remark, we collect other cases where Assumption \ref{ass:noise-matrix}, and hence Theorem \ref{thm:main}, hold.
\begin{remark}
  Assumption \ref{ass:noise-matrix} 
  %The conclusion of Corollary \ref{cor:main} 
  holds under various relaxed assumptions on the noise matrix $E_N$, which we list below.
  \begin{enumerate}
    \item  
%      Assumption \ref{ass:noise-matrix} holds
      When the entries of $E_N$ are independent and dominated by a single distribution (in the Fourier-analytic sense) that has a $\kappa$-controlled 
      second moment for some $\kappa >0$, see \cite[Definition 2.2 and
      Remark 2.8]{TV08}.
     % (For the definition the latterof a random variable with a $\kappa$-controlled second moment we refer the reader to \cite[Definition 2.2]{TV08}.
    \item
When the entries of $E_N$ are independent, satisfy a uniform 
anti-concentration bound near $0$, and have uniform lower bound on the truncated variance, see \cite[Lemma A.1]{BC}. 
Furthermore, \cite[Theorem 2.9]{TV08} and \cite[Lemma A.1]{BC} allow $E_N$ to be a sparse random matrix. 
\item
When the entries of $E_N$ have an inhomogeneous variance profile satisfying 
appropriate assumptions, by a recent result of Cook \cite{cook}.
Specifically, by \cite[Theorem 1.24]{cook},
the assumption
is satisfied 
%Corollary \ref{cor:main} extends to the case 
when the variance profile is 
{\em super-regular}, see \cite[Definition 1.23]{cook} for a precise formulation.
\item 
  When $E_N=\sqrt{N} U_N$, where $U_N$ is a Haar distributed unitary matrix,
  see \cite[Theorem 1.1]{RV-JAMS}.
  \end{enumerate}
\end{remark}
\begin{remark} We believe
  that the sequence $N^{-\gamma}$ in Theorem \ref{thm:main}
  can be replaced by any sequence
  $\mathfrak{a}_N$ satisfying $\sqrt{N} \mathfrak{a}_N\to_{N\to\infty} 0$. We chose 
  to work with $N^{-\gamma}$ in order to somewhat simplify the proofs.
\end{remark}
\begin{remark}
  A general notion developed to deal with perturbations of non-normal matrices is that of pseudospectrum, see \cite{trefethen1} for an extensive review. 
  This notion provides worse-case estimates and does not focus on the evaluation of  limits of empirical measures under random perturbation. However,
    Theorem \ref{thm:main} is consistent
    with predictions based on pseudospectrum.
    For a thorough discussion of how pseudospectrum relates
    to Theorem \ref{thm:main}, see \cite[Section 1.3]{BPZ}
    and \cite{trefethen}.
\end{remark}
Our approach to the proof of Theorem \ref{thm:main} differs from the one
employed in \cite{BPZ}, which derived a deterministic equivalence that worked only for complex i.i.d.~Gaussian perturbations (in particular, even real
Gaussian perturbations are not covered by \cite{BPZ}). Instead, our 
approach is based on a perturbation idea that can be traced back in this context to \cite{GWZ}. See Section \ref{sec:discussion} below for a further discussion on this. 

To describe the approach of this paper we first 
recall the important notion of 
logarithmic potential associated with a probability measure $\mu$.
\begin{dfn}[Log-potential]
For a probability measure $\mu$ supported on the complex plane define its log-potential as follows:
\[
\mathcal{L}_\mu(z):= \int \log |z - x| d\mu(x), \qquad z \in \C.
\]
\end{dfn}

%To prove Theorem \ref{thm:main} 
As a first step,
we will show that there exists a random matrix $\Delta_N$, 
with a polynomially decaying spectral norm, 
such that the conclusion of Theorem \ref{thm:main} holds
with $N^{-\gamma}E_N$ replaced by $\Delta_N$. 
%Then the proof of Theorem \ref{thm:main} follows from a general replacement result. First let us state the result when the noise matrix is $\Delta_N$. 
\begin{theorem}\label{thm:existence-of-Delta}
Let $T_N$ be any $N \times N$ Toeplitz matrix with a symbol ${\bm a}$, where ${\bm a}$ is a Laurent polynomial. Then, there exists a random matrix $\Delta_N$ with 
\begin{equation}\label{eq:E-op-norm}
\P(\| \Delta_N\| \ge N^{-\gamma_0}) =o(1),
\end{equation}
for some $\gamma_0 >0$, so that $L_N^{T+\Delta}$
%the empirical 
%distribution of the eigenvalues of $T_N+\Delta_N$ 
converges weakly, in probability, to  $\mu_{\bm a}$. 
Equivalently, for Lebesgue almost every $z\in\mathbb{C}$,
${\mathcal L}_{L_N^{T+\Delta}}(z)\to {\mathcal L}_{\mu_{\bm a}}(z)$, in probability.
\end{theorem}
\begin{remark}
  \label{rem-twisted}
  We do not know the analogue of Theorem \ref{thm:existence-of-Delta}
  for the \textit{twisted} Toeplitz matrices considered in \cite{BPZ}, and their
  non-triangular generalizations. For this reason, we cannot extend 
  Theorem \ref{thm:main} to the general banded
  twisted case. See however Remark \ref{rem-ell}
  below for the case of \textit{upper triangular}
  twisted Toeplitz matrices.
\end{remark}
We next state the replacement principle alluded to above.
Here and in the sequel, 
$B_\C(c,R)$ denotes the open ball in the complex plane of center $c$ and radius $R$.

\begin{theorem}[Replacement principle]\label{thm:GWZ}
Let $A_N$ be any deterministic matrix with a bounded operator norm. Suppose $\Delta_N$ and $E_N$ are random matrices. Let $\mu$ be a 
probability measure on $\mathbb C$ whose support is contained in
%that is compactly supported on the complex plane. That is, the support of $\mu$ is contained in 
$B_\C(0,R_0/2)$ for some $R_0< \infty$.
Assume the following.
\begin{enumerate}

\item[(a)] $E_N$ and $\Delta_N$ are independent. $\Delta_N$ satisfies \eqref{eq:E-op-norm} and $E_N$ satisfies Assumption \ref{ass:noise-matrix}.  
%\item[(b)] For Lebesgue a.e.~$z \in B_\C(0,R_0/2)$,
%\begin{equation}\label{eq:regularity}
%\lim_{\vep \downarrow 0} \int \log |x-z| {\bf 1}_{\{|x-z| \le \vep\}} d\mu(x) =0.
%\end{equation}
\item[(b)] For Lebesgue a.e.~$z \in B_\C(0,R_0)$, the empirical distribution of the singular values of $A_N -z \Id_N$ converges weakly, to the law induced by $|X- z|$, where $X \sim \mu$.
\item[(c)] 
  For Lebesgue a.e.~every $z \in B_\C(0,R_0)$, 
  %Further assume that
\begin{equation}\label{eq:log-pot-conv-1}
\mathcal{L}_{L_N^{A+\Delta}}(z) \to \mathcal{L}_\mu(z), \qquad \text{ as } N \to \infty, \text{ in probability}.
\end{equation}
\end{enumerate}
Then, for any $\gamma >\frac12$, for Lebesgue a.e.~every $z \in B_\C(0,R_0)$,
\begin{equation}\label{eq:conv-log-pot}
\mathcal{L}_{L_N^{A+N^{-\gamma}E}}(z) \to \mathcal{L}_\mu(z), \qquad \text{ as } N \to \infty, \text{ in probability}.
\end{equation}
\end{theorem}
Theorem \ref{thm:GWZ} is a generalization of the replacement lemma in 
\cite[Theorem 5]{GWZ}, with the advantage that it allows for more general noise
models $E_N$ and that it
is stated directly in terms of logarithmic potentials and avoids the need to realize the $*$-limit
of $A_N$ as a regular element of a non-commutative probability space.
It may be of independent interest beyond the study of perturbations of
Toeplitz matrices.
\begin{remark}
  \label{rem-ell}
  Theorem \ref{thm:GWZ} shows that \cite[Theorem 4.1]{BPZ} remains true 
  if one replaces there the complex Gaussian noise $G_N$ by a noise
  $E_N$ satisfying Assumption \ref{ass:noise-matrix}. This can be seen by 
  using in Theorem \ref{thm:GWZ}
  $\Delta_N=N^{-\gamma}G_N$, and using \cite[Lemma 4.6]{BPZ} 
  to verify condition (b) of the theorem.
\end{remark}

\subsection{Related results and extensions}\label{sec:discussion}
The study of the limiting \abbr{ESD} of random perturbations of Toeplitz matrices can be traced back to \cite{DH} where in the simplest case of ${\bm a}(\lambda)=\lambda$, i.e.~when the Toeplitz matrix is the standard Jordan matrix, they derive the limit by studying a relevant {\em Grushin problem}. On the other hand \cite{GWZ} derives the limit in the same set-up by first analyzing the limit of the log-potential of the \abbr{ESD} of a specific (deterministic) perturbation of the Jordan matrix. Then they use an argument similar in spirit to that of Theorem \ref{thm:GWZ} which allows them to replace that specific perturbation by a polynomially vanishing Gaussian perturbation. When the Toeplitz matrix is non-triangular with an  arbitrary symbol it is not straightforward to find the required perturbation. Furthermore, it is not clear whether there exists at all some deterministic perturbation allowing one
to apply \cite[Theorem 5]{GWZ}. Theorem \ref{thm:existence-of-Delta} of this paper shows that one can indeed find a {\em random} perturbation which does that job. Moreover, instead of appealing to \cite[Theorem 5]{GWZ} we use Theorem \ref{thm:GWZ} which enables us to consider a broad class of random perturbations.  

Recently in \cite{BPZ} the limiting spectral distribution of Gaussian perturbation of triangular Toeplitz matrices has been derived by adopting a different strategy. The key to the proof in \cite{BPZ} lies in the following observation: If for Lebesgue a.e.~$z \in \C$ the number of polynomially small singular values of $M_N - z \Id_N$ is not too large, where $\{M_N\}_{N \in \N}$ is some sequence of matrices and $\Id_N$ is the identity matrix, then the limiting \abbr{ESD} of Gaussian perturbations of $M_N$ can be described by the {\em Brown measure} associated with the limiting operator. So it boils down to finding estimates on the number of small singular values. When $M_N=T_N$, a triangular Toeplitz (or a twisted Toeplitz) matrix, this task has been accomplished in \cite{BPZ}. If $T_N$ is a non-triangular matrix then the approach to finding bounds on the number of small singular values that is used in \cite{BPZ} fail.

Let us add that recent works of Sj\"{o}strand and Vogel \cite{SV, SV1} also deal with the limiting spectrum of Gaussian perturbations of general Toeplitz matrices. They use yet another strategy which is similar in spirit to the one adopted in \cite{DH}. In particular, their methods are robust enough that in \cite{SV1} they apply them to Toeplitz operators with unbounded symbols.

There are several possible extensions of this paper that one can pursue. For example, one may be interested in understanding finer details of the spectrum, such as the behavior of the outliers of random perturbations of Toeplitz matrices. Building on the ideas of this paper the behavior of the outliers has been studied in a follow-up work \cite{BZ}. 

Another interesting question would be to study  the limiting \abbr{ESD} of random perturbations of Toeplitz matrices with infinite symbols; as mentioned above, for certain
perturbations this was achieved in \cite{SV1}. A careful inspection of the proof of Theorem \ref{thm:main} of this paper reveals that one can build on the strategies developed in this paper to consider the case of Toeplitz matrices with a slowly growing bandwidth. For ease of writing and explanation we chose to work with a fixed bandwidth. The case of Toeplitz matrices
 with a general infinite symbol is at present beyond the scope of our methods.
% one will need new ideas. 

\subsection*{Outline of the rest of the paper} We will show in Section \ref{sec:proof-thm-main} that Theorem \ref{thm:main} is an immediate consequence of Theorems \ref{thm:existence-of-Delta} and \ref{thm:GWZ}. In Section \ref{sec:proof-outline} we provide the outlines of the proofs of Theorems \ref{thm:existence-of-Delta} and \ref{thm:GWZ}. The proofs of these two theorems are carried out in Sections \ref{sec:proof-existence-of-Delta} and \ref{sec:proof-GWZ}, respectively. Appendix \ref{sec:alg-results} contains some algebraic results are that are used in the proofs. 

\subsection*{Acknowledgements} AB is partially supported by a Start-up Research Grant (SRG/2019/001376) from Science and Engineering Research Board of Govt.~of India, and ICTS--Infosys Excellence Grant. OZ is partially supported by Israel Science Foundation grant 147/15 and funding from the European Research Council (ERC) under the European Union’s Horizon 2020 research and innovation program (grant agreement number 692452). We thank the anonymous referees for helpful comments that enhanced the presentation of this paper.

\section{Outlines of proofs of Theorems \ref{thm:existence-of-Delta} and \ref{thm:GWZ}}\label{sec:proof-outline}

We begin with
an outline of the proof of Theorem \ref{thm:existence-of-Delta}. 
From \cite[Theorem 2.8.3]{tao2012topics} and the fact that the support of 
$\mu_{\bm a}$ is compact, 
it 
suffices to show that for Lebesgue a.e.~$z$ in some large compact subset of the complex plane, $\mathcal{L}_{L_N^{T+\Delta}}(z) 
\to \mathcal{L}_{\mu_{{\bm a}}}(z)$ in probability. Toward this goal, it is useful first to obtain a different representation of the limit. 

%, where $\mu_{\bm a}$ is the law of ${\bm a}(U)$ with ${\bm a}$ and $U$ as in Theorem \ref{thm:existence-of-Delta}. 
%To derive the above convergence we first need to identify the limit. This is done in the following lemma.

\begin{lemma}\label{lem:limit-log-potential}
Let 
\[
{\bm a}(\lambda)= \sum_{k=-d_2}^{d_1} a_k \lambda^k,
\]
for some $d_1, d_2 \in \N$. For any $z \in \C$ let $\lambda_1(z), \lambda_2(z),\ldots, \lambda_d(z)$ be the roots of the polynomial equation 
\begin{equation}\label{eq:eq-symbol}
P_{z,{\bm a}}(\lambda):=({\bm a}(\lambda)-z)\cdot \lambda^{d_2}=0,
\end{equation}
where $d:=d_1+d_2$. Then, for any $z \in \C$,
\[
\mathcal{L}_{\mu_{\bm a}}(z)= \log |a_{d_1}|+ \sum_{k=1}^d \log_+ |\lambda_k(z)|,
\]
where for $x \ge 0$, $\log_+(x) := \max\{\log x, 0\}$.
\end{lemma}
The proof of Lemma \ref{lem:limit-log-potential} is
a straightforward modification of 
that of \cite[Lemma 4.3]{BPZ}. We omit the details.

We next sketch the proof of Theorem \ref{thm:existence-of-Delta}, 
in the special case where 
$T_N$ is the Toeplitz matrix with symbol ${\bm a}(\lambda)=\lambda+\lambda^2$. 
%For brevity, for any $z \in \C$, let us denote 
Set $T_N(z):=T_N - z \Id_N$ where $z\in\C$. By Lemma 
\ref{lem:limit-log-potential}, the form of limiting log potential
depends on the number of roots of the polynomial $P_{z,{\bm a}}(\lambda)$ 
greater than one in modulus. This yields
(open) regions $\mathcal{R}_\ell \subset \C$, $\ell=0,1,2$, whose
boundaries have zero Lebesgue measure and 
the closure of whose union is  $\C$, so that
for all $z \in {\mathcal R}_\ell$ there are exactly $\ell$ roots of the equation $\lambda+\lambda^2 -z =0$ that are greater than one in modulus. Thus, to establish Theorem \ref{thm:existence-of-Delta} we need to find a noise matrix $\Delta_N$ such that the following holds:
\begin{equation}\label{eq:limit-heuristic}
\lim_{N \to \infty} \mathcal{L}_{L_N^{T+\Delta}}(z) = \lim_{N \to \infty} \frac{1}{N}\log |\det(T_N(z) + \Delta_N)| = \left\{\begin{array}{ll} \log|\lambda_1(z)| + \log|\lambda_2(z)| & \mbox{ if } z \in \mathcal{R}_2\\
\log|\lambda_1(z)| & \mbox{ if } z \in \mathcal{R}_1\\
0 & \mbox{ if } z \in \mathcal{R}_0
\end{array}
\right.,
\end{equation}
where $\lambda_1(z)$ and $\lambda_2(z)$ are the roots of the relevant equation arranged in the non-increasing order of their moduli. We refer the reader to Figure \ref{fig:limacon-shade} for an illustration of the regions $\mathcal{R}_\ell,\, \ell=0,1,2$.
\begin{figure}
  \includegraphics[height=2.5in, width=2.5in]{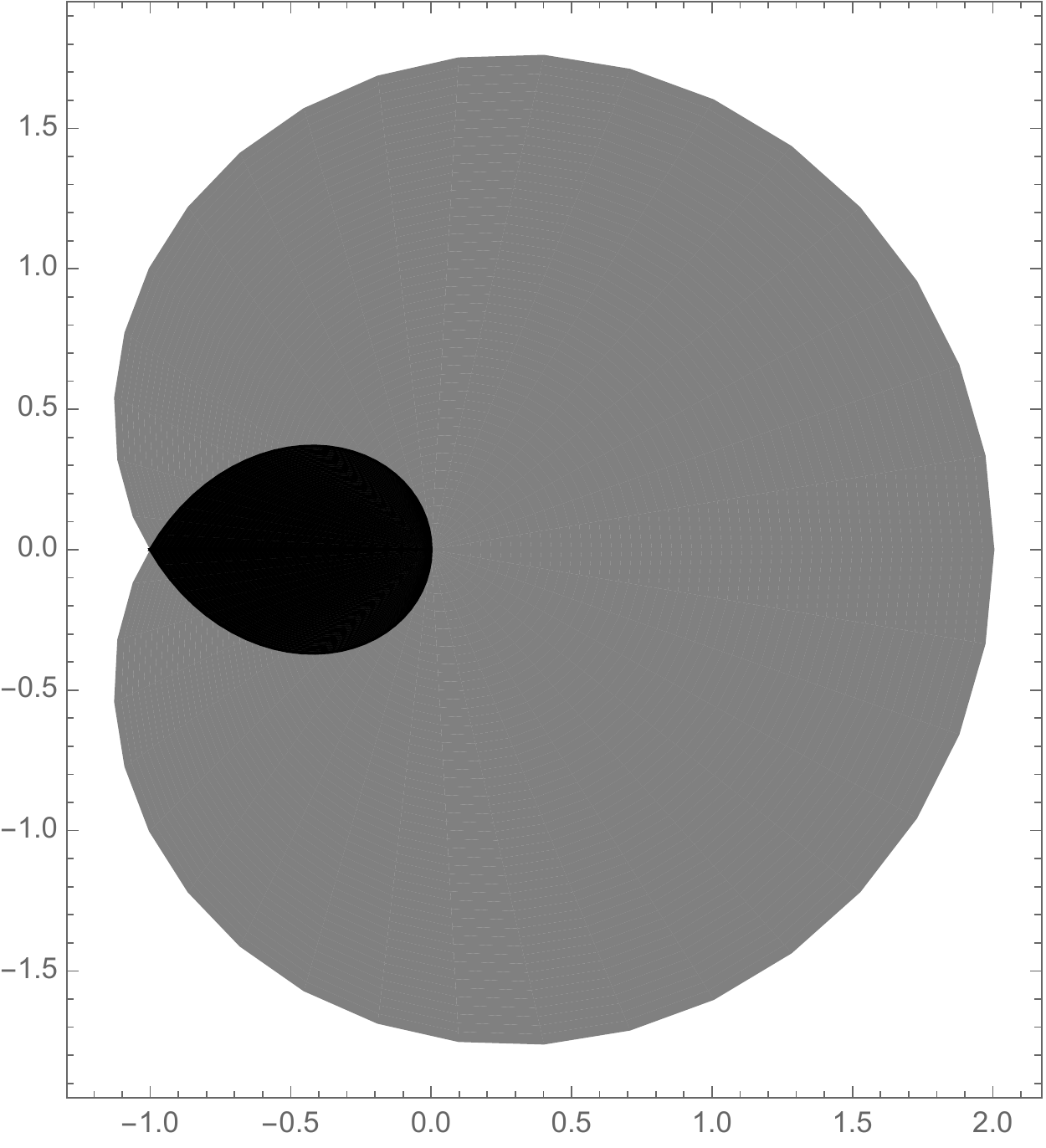}
  \caption{The open regions $\mathcal{R}_i$ for
  the polynomial $\lambda \mapsto {\bm a}(\lambda) := \lambda+\lambda^2$,
  with $\mathcal{R}_0$ in black, $\mathcal{R}_1$ in 
  grey and $\mathcal{R}_2$ in white. 
  Inside $\mathcal{R}_\ell$, there are exactly
  $\ell$ roots 
  %the \texttt{black} and \texttt{gray} regions (excluding their boundaries), denoted  by $\mathcal{R}_0$ and $\mathcal{R}_1$, none and one of the roots 
  of the equation ${\bm a}(\lambda)=z$ that are greater than one in moduli.}
  %, respectively. In $\mathcal{R}_2:= \C \setminus \overline{\mathcal{R}_1 \cup \mathcal{R}_0}$ both roots of ${\bm a}(\lambda)=z$ are greater than one in moduli.}
  \label{fig:limacon-shade}
\end{figure}
(We will see later that it is enough to consider the noise $\Delta_N$
supported on the lower left elements $\Delta_N(N,1), 
\Delta_N(N,2), \Delta_N(N-1,1)$.)

To derive \eqref{eq:limit-heuristic},
we expand the determinant of $T_N(z)+\Delta_N$. 
The latter 
can be written as a linear combination of products of 
determinants of various sub-matrices of $T_N(z)$ and $\Delta_N$ 
(see Lemma \ref{lem:cauchy-binet} below). We identify 
%Therefore, it remains to identify 
the dominant term in this expansion, as follows. 
Let $A_N[X;Y]$ denote
the sub-matrix of $A_N$ induced by the rows and the columns indexed 
by $X$ and $Y$, respectively. 
Recalling Widom's theorem concerning
the determinant of a finitely banded Toeplitz matrix (see \cite{SS, widom}), we obtain
%which gives
%to deduce that  
\begin{equation}\label{eq:widom}
\det(T_N(z)[X;Y]) \sim \left\{ \begin{array}{ll}
|\lambda_1(z)|^N \cdot |\lambda_2(z)|^N & \mbox{ if } X= Y=[N]\\ 
|\lambda_1(z)|^N & \mbox{ if } X=[N-1], \, Y = [N]\setminus \{1\}\\
1 & \mbox{ if } X=[N-2], \ Y = [N]\setminus \{1,2\}
\end{array}
\right.,
\end{equation}
where we write $a_N \sim b_N$ to indicate that there exists some absolute constant $C>0$ such that $N^{-C} a_N \le b_N \le N^C a_N$, for all large $N$. 

From \eqref{eq:widom} we see that if $z \in \mathcal{R}_0$ or $\mathcal{R}_1$ 
then there are sub-matrices of $T_N(z)$ whose determinants
are of larger magnitude than that of $T_N(z)$. 
We also note that the expansion of $\det (T_N(z)+\Delta_N)$ has
terms that are products of determinants of these sub-matrices 
and the determinant of relevant sub-matrices of the noise matrix $\Delta_N$
(of fixed dimension), 
where the latter can be chosen to be non-zero 
and only polynomially (in $N$) decaying. It follows that
if the determinants of those sub-matrices of $T_N(z)$
are of maximal exponential growth among the determinants of all possible 
sub-matrices of $T_N(z)$,
then $\frac{1}{N}\log |\det(T_N(z)+\Delta_N)|$ 
converges to the limit in \eqref{eq:limit-heuristic}. 
This not only explains  how the limit arises but also identifies potential 
candidates for the dominant terms 
(depending on the location of $z$ in the complex plane) in the expansion of the determinant, and gives a heuristic for the proof of
Theorem \ref{thm:existence-of-Delta}.

%is motivated by this heuristic. However, as we explain below, obtaining a rigorous proof based on this heuristic involves a few critical steps. 

To justify this heuristic and obtain an actual
proof of Theorem \ref{thm:existence-of-Delta} in the case under consideration,
it is natural to extend \eqref{eq:widom} and claim that  
\begin{equation}\label{eq:widom-1}
\sum_{\ell \ne k} P_\ell(z) =o\left(N^{-C}\cdot \prod_{j=1}^k |\lambda_j(z)|^N\right) =\Omega(|P_k(z)|), \qquad z \in \mathcal{R}_k, \, k =0,1,2,
\end{equation}
for some large absolute constant $C$, with large probability, where $P_\ell(z)$ is the homogeneous polynomial of degree $\ell$ in the entries of $\Delta_N$, in the expansion of the determinant of $T_N (z)+\Delta_N$. \corA{In \eqref{eq:widom-1} we have used the standard notations $a_n=o(b_n)$ and $a_n = \Omega(b_n)$ to denote $\lim_{n \to \infty} a_n/b_n=0$ and $\liminf_{n \to \infty} a_n/b_n >0$, respectively.}
 Finding bounds on $P_\ell(z)$ requires the same for $\det( T_N(z)[X;Y])$ for all subsets $X, Y \subset [N]$ such that $|X|=|Y|=N-\ell$. As $T_N(z)[X,Y]$ is not necessarily a Toeplitz matrix for arbitrary choices of $X, Y \subset [N]$ we can no longer rely on Widom's result. We overcome this obstacle by noting that any upper triangular finitely banded Toeplitz matrix $T_N$ can be represented as a product of bidiagonal matrices, where the bidiagonal matrices depend on the roots of polynomial equation associated with the symbol of the Toeplitz matrix in context. Since the determinant of any sub-matrix of a bidiagonal matrix is easily computable (see Lemma  \ref{lem:bidiagonal-det-1}) one can then use the
Cauchy-Binet theorem to find a bound on $\det (T_N(z)[X;Y])$. 
Using this and some combinatorial arguments, we then obtain
the desired bound on $P_\ell(z)$ whenever the entries of $\Delta_N$ are 
uniformly polynomially vanishing.  

We emphasize that the approach described above generalizes easily to
triangular
finitely banded Toeplitz matrix. The general case requires a modification, 
since
non-triangular Toeplitz matrices cannot be decomposed into a product.
%do not have the product structure. 
We resolve this issue
by using the following simple key observation:~any Toeplitz matrix with 
finite symbol can be viewed as a sub-matrix of an upper triangular 
Toeplitz matrix with an another finite symbol of a slightly larger dimension. 
Using this observation, we can then follow the same scheme as described above to find an upper bound on $P_\ell(z)$.

To complete the proof of \eqref{eq:widom-1} we then need 
to find a lower bound on the predicted dominant term, $P_k(z)$. This
is obtained using
an anti-concentration estimate,
which is shown to hold whenever the entries of $\Delta_N$ 
are assumed to have a bounded density, which we will impose since the matrix
$\Delta_N$ is an auxilliary matrix and does not appear in the statement of
our main theorem, Theorem \ref{thm:main}. See Lemma \ref{lem:dom-term} 
and Proposition \ref{prop:anti-conc}. This will prove \eqref{eq:widom-1}. 
To finish the proof of Theorem \ref{thm:existence-of-Delta}, we then obtain an (easy)
matching upper bound on $P_k(z)$.

We next outline the proof of Theorem \ref{thm:GWZ}. It suffices to show that for Lebesgue a.e.~$z$ in a compact subset of $\C$,
\begin{equation}\label{eq:log-pot-diff}
\lim_{N \to \infty} \left| \mathcal{L}_{L_N^{A+\Delta}}(z) - \mathcal{L}_{L_N^{A+N^{-\gamma} E}}(z)\right| =0, \qquad \text{ in probability}.
\end{equation}
Using the
assumptions of Theorem \ref{thm:GWZ} and standard perturbation results 
for the spectrum of Hermitian matrices, it readily follows that 
$\nu_{A_N+\Delta_N}^z$ and $\nu^z_{A_N+N^{-\gamma} E_N}$, 
the empirical distributions of the singular values of 
$A_N(z)+\Delta_N$ and $A_N(z)+N^{-\gamma} E_N$, respectively, 
have the same limit, and that limit is $\mu_z$, the law of $|X-z|$ where
$X\sim \mu$. As $\log(\cdot)$ is unbounded both near $0$ and 
$\infty$, the limit in \eqref{eq:log-pot-diff} is not immediate from this. 
Using bounds on the Hilbert-Schmidt norms of the relevant matrices 
the singularity near $\infty$ can be taken care of. 
Treating the singularity of $\log(\cdot)$ near $0$ involves two steps. 
As the integral of $\log(\cdot)$ near zero,  with respect to $\mu_z$
%, where $\mu_z$ is the law of $|X-z|$, $X \sim \mu$, and $\mu$ is the limit measure,
is negligible, using assumptions (b)-(c) the same can be shown to hold for 
$\nu^z_{A_N+\Delta_N}$. Hence, it suffices to show that the integral of 
$\log(\cdot)$ on the interval $(0,\vep)$ with respect to 
$\nu^z_{A_N+N^{-\gamma} E_N}$ goes to zero as $\vep \downarrow 0$.

The latter is obtained by standard arguments, as follows.
We use Assumption \ref{ass:noise-matrix}(ii) 
to deduce that it is enough to integrate 
$\log(\cdot)$ in $(N^{-\kappa_\star}, \vep)$ for some small constant 
$\kappa_\star$. Now, using bounds on Hilbert-Schmidt norms of $E_N$ 
and $\Delta_N$ one can derive a bound on the difference of the 
Stieltjes transforms 
of $\nu^z_{A_N+N^{-\gamma} E_N}$ and $\nu^z_{A_N+\Delta_N}$. 
Using this, one obtains
that the difference of the total mass
of any interval near zero, 
under $\nu^z_{A_N+N^{-\gamma} E_N}$ and $\nu^z_{A_N+\Delta_N}$,
is negligible. Upon using an integration by parts, this gives the 
required control on
the integral of $\log(\cdot)$ near $0$ under $\nu^z_{A_N+N^{-\gamma}E_N}$ and
completes the proof.
%. Recalling that the integral of $\log(\cdot)$ near $0$ under $\nu^z_{A_N+\Delta_N}$ is negligible, the outline of the proof completes. 

%As our goal is to study the behavior of the outliers we need to study the roots of the equation \(\det (\cM_N(z))=0\) for $z \in \C$. Note that the determinant is a complicated polynomial in $z$ and the entries of $G_N$ of degree $N$. Hence, to find the interior density of the outliers it will be useful to identify the dominant term in the expansion of the determinant so that upon applying Rouch\'{e}'s theorem it will suffice to analyze the roots of only the dominant term. 
%This necessitates the following notations. For $k \in [N]$ we denote

%\red{Using \cite[Eqn.~(4)]{FPZ} we observe that $P_k(z)$ is the homogeneous polynomial of degree $k$ in the entries of $G_N$ obtained from the expansion of the determinant of $\cM_N(z)$. } 

\section{Proof of Theorem \ref{thm:main} using Theorems \ref{thm:existence-of-Delta} and \ref{thm:GWZ}}\label{sec:proof-thm-main}
We will take $\Delta_N$ provided by Theorem \ref{thm:existence-of-Delta},
set
$\mu=\mu_{\bm a}$ in Theorem \ref{thm:GWZ},
and verify that the hypotheses of the latter hold.
Clearly, $A_N$ has uniformly bounded operator norm.
%To prove Theorem \ref{thm:main} we only need to verify the assumptions of Theorem \ref{thm:GWZ}. 
The assumption (a) is obvious. 
To see that
assumption (b) holds, it is enough  to check that for $k$ positive integer,
\begin{equation}
  \label{eq-moments}
  \lim_{N\to\infty} \frac1N\mbox{\rm tr} ((z\Id_N-T_N)(z\Id_N-T_N)^*)^k-
  \E (|z-\sum_{i=-d_2}^{d_1} a_i U^i)|^{2k})=0.
\end{equation}
To check \eqref{eq-moments} we first note that 
\begin{equation}\label{eq:T_N-decomp}
T_N = \sum_{m=0}^{d_1} a_m J_N^m + \sum_{n=1}^{d_2} a_{-n} (J_N^*)^n,
\end{equation}
where $J_N$ is the nilpotent matrix given by given by $(J_N)_{i,j} = {\bf 1}_{j = i+1}$. Using this observation we then expand ${\rm tr} ((z\Id_N-T_N)(z\Id_N-T_N)^*)^k$ and find out the limit of each term in term in that expansion. To work out this step we need to introduce some notation. 

Let 
$\underline m:=(m_1,\ldots,m_{2k})$ and $\underline n:=(n_1,\ldots,n_{2k})$ 
with $n_i,m_i$ non-negative integers bounded by $\max(d_1,d_2)$, and set $M:=
M_{\underline m,\underline n}:=\sum m_i-\sum n_i$. We say that $(\underline m,\underline n)$ is balanced if $M_{\underline m,\underline n}=0$. Using \eqref{eq:T_N-decomp} we find that 
\[\frac1N\mbox{\rm tr} ((z\Id_N-T_N)(z\Id_N-T_N)^*)^k=
  \frac1N 
  \sum_{\underline m, \underline n} b_{\underline m,\underline n}(z)
 \mbox{\rm tr} \left[J_N^{m_1}(J_N^*)^{n_1}\cdots  J_N^{m_{2k}}(J_N^*)^{n_{2k}}\right]\,\]
for appropriate coefficients $b_{\underline m,\underline n}(z)$, while
\[\E (|z-\sum_{i=-d_2}^{d_1} a_i U^i)|^2)=\sum_{\underline m, \underline n}
b_{\underline m,\underline n}(z)\E [U^{M_{\underline m,\underline n}}],\]
with the same coefficients $b_{\underline m,\underline n}(z)$. Note that
$$\mbox{\rm tr} \left[J_N^{m_1}(J^*_N)^{n_1}\cdots  J^{m_{2k}}_N(J^*_N)^{n_{2k}}\right]=0$$
if $(\underline m,\underline n)$ is not balanced, while
$$\lim_{N\to \infty} \frac1N
\mbox{\rm tr} \left[J_N^{m_1}(J_N^*)^{n_1}\cdots  J_N^{m_{2k}}(J_N^*)^{n_{2k}}\right]=1$$
if $(\underline m,\underline n)$ is  balanced. Similarly,
$\E [U^{M_{\underline m,\underline n}}]$ equals $1$
if $(\underline m,\underline n)$
is  balanced, and vanishes otherwise.
Combining these facts, we obtain \eqref{eq-moments}, and thus verify
that assumption (b) holds.

Assumption (c) holds because, from
%Returning to the proof of Theorem \ref{thm:main}, we have from
%we note that 
Theorem \ref{thm:existence-of-Delta}, we see that
for 
Lebesgue a.e.~$z \in B_\C(0,R_0)$,
\[
\mathcal{L}_{L_N^{T+\Delta}}(z) \to \mathcal{L}_{\mu_{\bm a}}(z), \qquad \text{ in probability}.
\]
%for Lebesgue a.e.~$z \in B_\C(0,R)$ for any fixed large $R < \infty$. 
%Therefore, the assumption (d) of Theorem \ref{thm:GWZ} is 
%also satisfied.

%, and 
%in particular, for almost every $z\in B_\C(0,R_0)$, 
%the empirical measure of singular values of $T_N+\Delta_N-z\Id_N$ converges to the law
%of $|X-z|$ where $X\sim \mu_{\bm a}$, in probability.
%Since the norm of  $\Delta_N$ goes to $0$ as $N\to\infty$ in probability
%while the operator norm of 
%$A_N$ is uniformly bounded, the same convergence (in probability) holds for the
%empirical measure of singular values of $T_N-zI_N$, implying that 
%assumption (c) holds. (See the beginning of the proof of Theorem
%\ref{thm:GWZ} for a similar argument and more details.)

We have checked all assumptions of Theorem \ref{thm:GWZ}; applying the latter
we conclude that
for Lebesgue a.e.~$z \in B_\C(0,R_0)$ and for
any $\gamma > \frac12$,
\[
\mathcal{L}_{L_N^{T+N^{-\gamma} E}}(z) \to \mathcal{L}_{\mu_{\bm a}}(z), \qquad \text{ in probability}.
\]
By the proof of \cite[Theorem 2.8.3]{tao2012topics} and the fact that
the support of $\mu_{\bm a}$ is compact,
this implies the convergence in probability of 
$L_N^{T+N^{-\gamma} E}$ to $\mu_{\bm a}$ in the vague \corA{topology}, 
and hence in the weak topology.
\qed

\section{Proof of Theorem \ref{thm:existence-of-Delta}}\label{sec:proof-existence-of-Delta}
In this section we prove Theorem \ref{thm:existence-of-Delta}. 
As outlined in Section \ref{sec:proof-outline},
the key is to establish \eqref{eq:widom-1}. 
Turning to this task, introduce,
for any $k \in [N]$,
\begin{equation}\label{eq:P_k-z}
P_k(z):= \sum_{\substack{X, Y \subset [N]\\ |X|=|Y|=k}} (-1)^{\sgn(\sigma_{X}) \sgn(\sigma_{Y})} \det (T_N(z)[{X}^c; {Y}^c]) \cdot \det (\Delta_N[X; Y]),
\end{equation}
where ${X}^c:=[N]\setminus X$, ${Y}^c:= [N]\setminus Y$, and for $Z \in \{X, Y\}$ $\sigma_Z$ is the permutation on $[N]$ which places all the elements of $Z$ before all the elements of ${Z}^c$, but preserves the order of the elements within the two sets. 
Define
\[
\mathscr{S}_{d_1,d_2}:= \left\{(i,j) \in [N]\times [N]: i-j \notin \{-(N-\ell); \, \ell=1,2,\ldots,d_1\} \cup\{N-\ell; \, \ell=1,2,\ldots, d_2\}\right\}.
\]
To prove Theorem \ref{thm:existence-of-Delta} we will choose $\Delta_N$ which satisfies the following band structure:
\[
(\Delta_N)_{i,j} \ne 0 \qquad \text{ only if } i-j \in 
\mathscr{S}_{d_1,d_2}^c,
\]
where 
$(\Delta_N)_{i,j}$ denotes the $(i,j)$-th entry of $\Delta_N$. That is,
$\Delta_N$ has non-zero entries only in its
lower left and upper right corners,
and the widths of those corners are determined by $d_1$ and $d_2$, 
respectively. As indicated in \eqref{eq:widom} such a band structure is necessary (as we will see it is also sufficient) to have a non-zero contribution from the sub-matrices of $T_N(z)$ 
whose determinants are of larger magnitudes compared to that of the 
whole matrix, in the expansion of $\det(T_N(z)+\Delta_N)$. 
%For ease of writing, let us denote
%\[
%\mathscr{S}_{d_1,d_2}:= \left\{(i,j) \in [N]\times [N]: i-j \notin \{-(N-\ell); \, \ell=1,2,\ldots,d_1\} \cup\{N-\ell; \, \ell=1,2,\ldots, d_2\}\right\}.
%\]
Recall from \eqref{eq:limit-heuristic}-\eqref{eq:widom} 
that the dominant term depends on the number of roots of 
$P_{z, {\bm a}}(\cdot)$ of \eqref{eq:eq-symbol}, that are greater than one in modulus. Hence, we split the complex plane into regions according
to the number of roots of $P_{z, {\bm a}}(\cdot)$ with modulus greater 
than one, using the following notation.

%Below we will see that the randomness of the outliers of $\cM_N$ at $z \in \C$ depend on its location in the complex plane. To identify these different regions in the complex plane we have the following definitions. 

Let $\{-\lambda_i(z)\}_{i=1}^d$, $d:=d_1+d_2$\footnote{Hereafter $\{-\lambda_\ell(z)\}$ will denote
the roots of the equation $P_{z,{\bm a}}(\lambda)=0$. This change in notation is adopted to avoid the unnecessary appearance of signs in the determinant of the sub-matrices of $J_N-\lambda \Id_N$.}, be the roots of the equation 
$P_{z, {\bm a}}(\cdot)=0$,
arranged so that
$|\lambda_1(z)| \ge |\lambda_2(z)| \ge \cdots \ge |\lambda_d(z)|$.
For $z \in \C$, let $d_0(z)$ denote the number of roots of the equation $P_{z, {\bm a}}(\cdot)=0$ that are greater than or equal to one in moduli. Fixing $R < \infty$, for $-d_2 \le \gd \le d_1$ we define
\[
\cS_{\gd}:= \left\{ z \in B_\C(0,R): d_1 - d_0(z) = \gd \text{ and }  |\lambda_{d_0(z)}(z)| > 1 > |\lambda_{d_0(z)+1}(z)| \right\}.
\]
Note that 
\[
B_\C(0,R) \setminus (\cup_{\ell=-d_2}^{d_1} \cS_{\ell}) \subset \{z \in B_\C(0,R): P_{z, {\bm a}}(\lambda)=0\text{ for some } \lambda \in \mathbb{S}^1\}.
\] 
If $P_{z, {\bm a}}(\lambda)=0$ for some $\lambda \in \mathbb{S}^1$ then we also have that
\[
z = \sum_{\ell=-d_2}^{d_1} a_\ell \lambda^\ell.
\]
Therefore $B_\C(0,R) \setminus (\cup_{\ell=-d_2}^{d_1} \cS_{\ell})$ is contained in a set of Lebesgue measure zero and hence it is enough to consider $z \in \cup_{\ell=-d_2}^{d_1}\cS_\ell$. Further let $\cN$ be the set of $z$'s for which $P_{z}(\cdot)$ admits a double root. It follows from \cite[Lemma 11.4]{bottcher-finite-band} that the cardinality of $\cN$ is at most finite. 

The next lemma
identifies the dominant term in the expansion of $\det(T_N(z)+\Delta_N)$.

 \begin{lemma}\label{lem:dom-term}
Fix $\gd$ such that $ -d_2 \le \gd \le d_1$. Let $\Delta_N$ be such that 
\[
(\Delta_N)_{i,j} = N^{-\gamma_\star} \updelta_{i,j} {\bf 1}_{\{(i,j) \in \mathscr{S}_{d_1,d_2}\}}, \qquad i,j \in [N],
\]
for some $\gamma_\star >d$, where $\{\updelta_{i,j}\}$ are uniformly bounded real valued independent random variables with uniformly bounded densities with respect to the Lebesgue measure. Then, for Lebesgue a.e.~$z \in \cS_\gd$, and any $\vep_0 >0$,
\begin{equation}\label{eq:dom-term}
\lim_{N \to \infty} \P \left( \frac{\left|P_{|\gd|}(z) \right|}{|a_{d_1}|^N \cdot \prod_{i=1}^{d_0(z)} |\lambda_i(z)|^{N}} \le N^{-({\gamma_\star |\gd|+\vep_0})}\right)
 =0,
 \end{equation}
 where an empty product by convention is set to one.
\end{lemma}
Lemma \ref{lem:dom-term} yields 
a lower bound on the order of the magnitude of the predicted dominant term in the expansion of $\det(T_N(z)+\Delta_N)$. Next we need to show
that the sum of the rest of the terms is of smaller order. To show this,
we split it
 into two sums: 
 $\sum_{\ell < |\gd|} P_\ell (z)$ and $\sum_{\ell > |\gd|} P_\ell(z)$. 
 The second sum
 will be shown to be polynomially small compared to the leading term, 
 whereas the first will be shown to be exponentially small.
 This is the content of the two following lemmas. 

 \begin{lemma}\label{lem:rouche-multi-gr-d_0}
Let $\gd, \Delta_N$, and $\gamma_\star$ be as in Lemma \ref{lem:dom-term}. Then, for Lebesgue a.e.~$z \in \cS_\gd$,
\begin{equation}\label{eq:rouche-multi-gr-d_0}
\lim_{N \to \infty}  \frac{N^{\gamma_\star |\gd|+\frac{\gamma_\star-d}{2}}\left|\sum_{k=|\gd|+1}^N P_k(z) \right|}{|a_{d_1}|^N \cdot \prod_{i=1}^{d_0(z)} |\lambda_i(z)|^{N}} =0.
 \end{equation}
\end{lemma}

\begin{lemma}\label{lem:rouche-multi-le-d_0}
Under the same set-up as in Lemma \ref{lem:rouche-multi-gr-d_0}, for Lebesgue a.e.~$z \in \cS_\gd$, we have 
\[
\lim_{N \to \infty} \frac{\left|\sum_{k=0}^{|\gd|-1} P_k(z) \right|}{|a_{d_1}|^N \cdot \prod_{i=1}^{d_0(z)} |\lambda_i(z)|^N \cdot (1-\bar \vep)^N}
=0,
 \]
 for some small constant $\bar \vep :=\bar \vep (z, {\bm a}) \in (0,1)$.
 %, depending only on $z$.
\end{lemma}
The 
proofs of Lemmas \ref{lem:rouche-multi-gr-d_0} and 
\ref{lem:rouche-multi-le-d_0}  are in Section \ref{sec:non-dominant},
while 
the proof of Lemma \ref{lem:dom-term} is postponed to 
Section \ref{sec:dominant}. To complete the proof of Theorem \ref{thm:existence-of-Delta}, we will also need an upper bound on the dominant term, which is contained in the next lemma, whose proof is deferred to Section \ref{sec:dominant}.

\begin{lemma}\label{lem:dom-term-ubd}
Under the same set-up as in Lemma \ref{lem:dom-term}, 
for Lebesgue a.e.~$z \in \cS_\gd$, there exists a constant $C_0$ depending
on $z$ and ${\bm a}$ only, so that
\[
\limsup_{N \to \infty} \frac{P_{|\gd|}(z)}{|a_{d_1}|^N \cdot \prod_{i=1}^{d_0(z)}|\lambda_i(z)|^N} \le C_0.
\]
%for some constant $C_0$, depending only on $z$. 
\end{lemma}

Equipped with these four lemmas, we now compete the proof of Theorem \ref{thm:existence-of-Delta}.

\begin{proof}[Proof of Theorem \ref{thm:existence-of-Delta}]
From the definition of $\Delta_N$ it follows that there are at most $d$ non-zero entries in each row of $\Delta_N \Delta_N^*$. Furthermore, each entry of $\Delta_N \Delta_N^*$ is at most $O(N^{-2\gamma_\star})$. Therefore,
by the Gershgorin circle theorem,
it follows that $\| \Delta_N\| = O(N^{-\gamma_\star})$,
establishing the desired property \eqref{eq:E-op-norm}. Next,
as in the proof of Theorem \ref{thm:main},
the weak convergence of $L_N^{T+\Delta}$ to 
$\mu_{{\bm a}}$ follows from the convergence,
for Lebesgue a.e.~$z \in B_\C(0,R_0)$, of the log-potentials:
\begin{equation}\label{eq:log-pot-T-Delta-limit}
\mathcal{L}_{L^{T+\Delta}_N}(z)= \frac{1}{N}\log |\det (T_N(z)+\Delta_N)| \to \mathcal{L}_{\mu_{\bm a}}(z), \qquad \text{ in probability}. 
\end{equation}
%, where $R<\infty$ is some sufficiently large constant depending on the diameter of the support of $\mu_{\bm a}$. 
To this end, recalling the definition of $P_k(z)$ from \eqref{eq:P_k-z} and applying Lemma \ref{eq:det_decomposition} we have that 
\begin{equation}\label{eq:det-split-1}
\frac{1}{N}\log |\det(T_N(z)+\Delta_N)| = \frac{1}{N} \log \left| \sum_{k=0}^N P_k(z) \right| = \frac{1}{N}\log |P_{|\gd|}(z)| + \frac{1}{N}\log \left|1 + \frac{\sum_{k \ne |\gd|}P_k(z)}{P_{|\gd|}(z)}\right|,
\end{equation}
for any integer $\gd$ between $-N$ and $N$. Setting $\vep_0=\frac{\gamma_\star - d}{4} >0$ in Lemma \ref{lem:dom-term} and combining Lemmas \ref{lem:dom-term}-\ref{lem:rouche-multi-le-d_0} we note that for Lebesgue a.e.~$z \in \cS_\gd$,
there exists an event of 
probability at least $1-o(1)$ such that, on that event, we have 
\[
\left|\frac{\sum_{k \ne |\gd|} P_k(z)}{P_{|\gd|}(z)} \right| \le N^{-\frac{\gamma_\star -d}{8}},
\]
for all large $N$. This in turn implies that 
\begin{equation}\label{eq:det-split-2}
 \frac{1}{N} \log \left|1 + \frac{\sum_{k \ne |\gd|}P_k(z)}{P_{|\gd|}(z)}\right| \to 0, \qquad \text{ in probability},
\end{equation}
for Lebesgue a.e.~every $z \in \cS_\gd$. 
Finally combining Lemmas \ref{lem:dom-term} and \ref{lem:dom-term-ubd} we 
obtain that for Lebesgue a.e.~every $z \in \cS_\gd$,
\begin{equation}\label{eq:det-split-3}
\frac{1}{N}\log P_{|\gd|}(z) \to \log|a_{d_1}| + \sum_{\ell=1}^{d_0(z)} \log |\lambda_\ell(z)| = \mathcal{L}_{\mu_{\bm a}}(z), \qquad \text{ in probability}.
\end{equation}
Combining \eqref{eq:det-split-1}-\eqref{eq:det-split-3} we now deduce \eqref{eq:log-pot-T-Delta-limit} for Lebesgue a.e.~$z \in \cS_\gd$ and any integer $\gd$ such that $-d_2 \le \gd \le d_1$. This completes the proof. 
\end{proof}

\subsection{Upper bound on non-dominant terms}\label{sec:non-dominant}
Recall from Section \ref{sec:proof-outline} that to establish 
bounds on the predicted non-dominant terms, one 
uses the fact any upper triangular Toeplitz matrix with a 
finite symbol can be expressed as a product of bidiagonal matrices. 
To use the same representation for a non-triangular 
Toeplitz matrix we view it as a sub-matrix of an upper 
triangular Toeplitz matrix of a slightly larger dimension. Toward this end, we introduce the folowing definition.

\begin{dfn}[Toeplitz with a shifted symbol]\label{dfn:toep-shifted}
Let $T_N$ be a Toeplitz matrix with finite symbol ${\bm a}(\lambda)=\sum_{\ell=-d_2}^{d_1} a_\ell \lambda^\ell$ and as before $d=d_1+d_2$. For $\bar d_1, \bar d_2 \in \mathbb{N}$ such that $\bar d_1 + \bar d_2=d$ and $z \in \C$,
set
$T_N(z; \bar d_1, \bar d_2) := T_N(z; \bar d_1, \bar d_2) ({\bm a})$ to be the $N \times N$ Toeplitz matrix with first row and column 
\[
(a'_{d_1 -\bar d_1}, a'_{d_1-\bar d_1+1}, \ldots, a'_{d_1}, 0, \ldots, 0) \qquad \text{ and } \qquad (a'_{d_1 - \bar d_1}, a'_{d_1 - \bar d_1-1}, \ldots, a'_{-d_2}, 0, \ldots, 0)^{\sf T},\] 
respectively, where $a'_j := a_j - z \cdot {\bf 1}_{\{j=0\}}$, $j=-d_2,-d_2+1,\ldots, d_1$. That is,
\[
T_{N}(z; \bar d_1, \bar d_2):=\begin{bmatrix}
a_{d_1-\bar d_1} &\cdots &a_0-z& \cdots & \cdots& a_{d_1}&0&\cdots& 0\\
\vdots& a_{d_1 -\bar d_1} & & a_0-z & && \ddots& & \vdots\\
%0 &0& a_{-d_2} & \ddots & a_0-z&\ddots&a_{d_1}& 0\cdots\\
a_{-d_2} &  &\ddots &&\ddots&&&\ddots & \vdots\\
0 & \ddots & & \ddots &&\ddots& & & a_{d_1}\\
\vdots&\ddots &\ddots  & & \ddots & &\ddots & & \vdots\\
\vdots &&\ddots &\ddots  && \ddots & & \ddots&\vdots \\
\vdots & & &\ddots & \ddots &  & \ddots&  &a_0-z\\
\vdots & & & &\ddots&  \ddots &  & \ddots &\vdots \\
 0 & \cdots & \cdots & \cdots&\cdots &0 & a_{-d_2} & \cdots & a_{d_1-\bar d_1}
\end{bmatrix}.
\]

\end{dfn}

From Definition \ref{dfn:toep-shifted},
it follows that
\[
T_N(z)=T_N(z;d_1,d_2) = T_{N+d_2}( z; d,0) [[N]; [N+d_2]\setminus [d_2]].
\]
Note that $T_{N+d_2}(z;d,0)$ is an upper triangular Toeplitz matrix. Since $\{-\lambda_\ell(z)\}_{\ell=1}^d$ are the roots of the equation $P_{z, {\bm a} }(\lambda)=0$ we obtain that
\[
T_{N+d_2}(z;d,0) = \sum_{\ell=0}^d (a_{\ell-d_2} - z \delta_{\ell, d_2}) J_{N+d_2}^\ell = a_{d_1} \prod_{\ell=1}^d (J_{N+d_2}+ \lambda_\ell(z) \Id_{N+d-2}),
\]
where \corA{we recall that} $J_n$ is the nilpotent matrix given by $(J_n)_{i,j} = {\bf 1}_{j=i+1}$, for $i,j \in [n]$. 

Hence, recalling the definition of $\{P_k(z)\}_{k=1}^N$ from  \eqref{eq:P_k-z}, applying the
Cauchy-Binet theorem, and writing $S+ \ell:= \{x+\ell, x \in S\}$ for any set of integers $S$ and an integer $\ell$, we obtain that 
\begin{multline}\label{eq:det-decompose-1}
P_k(z)  \\
= \sum_{\substack{X, Y \subset [N]\\ |X|=|Y|=k}} (-1)^{\sgn(\sigma_{X}) \sgn(\sigma_{Y})} \det (T_{N+d_2}(z; d,0)[{X}^c; {Y}^c+d_2]) \det (\Delta_N[X; Y])\\
 =  \sum_{\substack{X, Y \subset [N]\\|X|=|Y|=k}}
 %\sum_{j=2}^{d-1} 
 \sum_{
 \substack{
 %\mbox{\small for $j=2,\ldots,d-1$}:\\
 X_j \subset [N+d_2], j=2,\ldots,d-1:\\ |X_j|=k+d_2}} (-1)^{\sgn(\sigma_{X}) \sgn(\sigma_{Y})}a_{d_1}^{N-k} \cdot \prod_{i=1}^d \det\left((J_{N+d_2} + \lambda_i(z)\Id_{N+d_2})[\COMP{X}_i; \COMP{X}_{i+1}]\right)\\
  \cdot \det (\Delta_N[X; Y]),
\end{multline}
where 
\begin{equation}\label{eq:X-Y-constraint}
X_1:= X_1(X):=X \cup [N+d_2]\setminus [N], \qquad X_{d+1}:=X_{d+1}(Y):= (Y+d_2) \cup [d_2],
\end{equation}
and $\COMP{Z}:=[N+d_2]\setminus Z$ for any set $Z \subset [N+d_2]$. Equipped with this preparatory decomposition of $P_k(z)$,
we are now ready to step into the proof of Lemma \ref{lem:rouche-multi-gr-d_0}.

\begin{proof}[Proof of Lemma \ref{lem:rouche-multi-gr-d_0}]
From the definition of the noise matrix it follows that the number of non-zero rows (and also non-zero columns) in $\Delta_N$ is at most $d$. This means that
$P_k(z)=0$ for any $k >d$. Therefore, it is enough to show that \eqref{eq:rouche-multi-gr-d_0} holds with the sum in the numerator being replaced by $P_k(z)$, where $ |\gd| < k \le d$.

To achieve this, we need to simplify \eqref{eq:det-decompose-1}; this simplification, summarized in 
\eqref{eq:P_k-decompose} and \eqref{eq:P_k-decompose-1} below,
will also be useful in the proof of Lemma \ref{lem:rouche-multi-le-d_0}.
From \eqref{eq:det-decompose-1}-\eqref{eq:X-Y-constraint},
we see that each $X_i$ is of cardinality $k+d_2$. Therefore, we write 
\[
X_i = \left\{ x_{i,1} < x_{i,2} < \cdots < x_{i,k+d_2} \right\}
\]
and for brevity we also denote $\cX_k:=(X_1, X_2, \ldots, X_{d+1})$. Applying Lemma \ref{lem:bidiagonal-det-1} we see that 
\[
\prod_{i=1}^d \det\left((J_{N+d_2} + \lambda_i(z)\Id_{N+d_2})[\COMP{X}_i; \COMP{X}_{i+1}]\right) \ne 0
\]
only when $\cX_k \in L_{{\bm \ell},k}$ for some ${\bm \ell}:=(\ell_1,\ell_2,\ldots,\ell_d)$ with $0 \le \ell_i \le N-k \le N+d_2$, $i=1,2,\ldots,d$,  where
\begin{multline*}
L_{{\bm \ell},k}:= \{ \cX_k:  \, 1 \le x_{i+1,1} \le x_{i,1} < x_{i+1,2} \le x_{i,2} < \cdots < x_{i+1,k+d_2} \le x_{i,k+d_2} \le N+d_2\\
\text{ and }  \, x_{i+1,1}+ \sum_{j=2}^{k+d_2}(x_{i+1,j}-x_{i,j-1})+ (N+d_2-x_{i,k+d_2}) =\ell_i+k+d_2, \text{ for all } i=1,2,\ldots,d\}.
\end{multline*}
Since
\[
x_{i+1,1}+  \sum_{j=2}^{k+d_2}(x_{i+1,j}-x_{i,j-1}) +  \sum_{j=1}^{k+d_2}(x_{i,j}-x_{i+1,j}) + (N+d_2- x_{i,k+d_2}) = N+d_2,
 \]
we have the following following equivalent representation of $L_{{\bm \ell}, k}$:
\begin{multline}\label{eq:L-ell-k}
L_{{\bm \ell},k}:= \{ \cX_k:  \, 1 \le x_{i+1,1} \le x_{i,1} < x_{i+1,2} \le x_{i,2} < \cdots < x_{i+1,k+d_2} \le x_{i,k+d_2} \le N+d_2,\\
 \, \sum_{j=1}^{k+d_2}(x_{i,j}-x_{i+1,j}+1)   =\hat \ell_i; \ i=1,2,\ldots,d_0,\\
\text{ and }  \, x_{i+1,1}+ \sum_{j=2}^{k+d_2}(x_{i+1,j}-x_{i,j-1})+ (N+d_2-x_{i,k+d_2}) =\hat \ell_i+k+d_2; \ i=d_0+1,d_0+2,\ldots,d\},
\end{multline}
where
\[
\hat \ell_i := \left\{\begin{array}{ll} \ell_i & \mbox{ if } i > d_0\\
N+d_2- \ell_i & \mbox{ if } i \le d_0.
\end{array}
\right.
\]
We also note that in \eqref{eq:det-decompose-1} the outer sum is over $X, Y \subset [N]$ and due to the constraint \eqref{eq:X-Y-constraint} we only need to consider $\cX_k \in \gL_{{\bm \ell}, k}$, where
\begin{equation}\label{eq:gL-ell-k}
\gL_{{\bm \ell}, k} := \{\cX_k \in L_{{\bm \ell}, k}: x_{1,k+j}=N+j; \ j  \in [d_2] \quad \text{ and } \quad x_{d+1,j}=j; \ j \in [d_2]\}.
\end{equation}
Thus applying Lemma \ref{lem:bidiagonal-det-1} again, from \eqref{eq:det-decompose-1} we now deduce that 
\begin{equation}\label{eq:P_k-decompose}
\frac{P_k(z)}{a_{d_1}^{N-k} \cdot \prod_{i=1}^{d_0}\lambda_i(z)^{N+d_2}}= \sum_{{\bm \ell}} Q_{{\bm \ell}, k}(z), 
\end{equation}
where 
\begin{equation}\label{eq:P_k-decompose-1}
Q_{{\bm \ell},k}(z):= \prod_{i=1}^{d_0} \lambda_i(z)^{-\hat \ell_i}  \prod_{i=d_0+1}^{d} \lambda_i(z)^{\hat \ell_i} \cdot \sum_{ \cX_k \in \gL_{{\bm \ell},k}}  (-1)^{\sgn(\sigma_{\mathbb{X}}) \sgn(\sigma_{\mathbb{Y}})}\det(\Delta_N[\mathbb{X}; \mathbb{Y}]),
\end{equation}
\[
 \mathbb{X}:=\mathbb{X}(X_1):= X_1 \cap [N], \qquad \text{ and } \qquad \mathbb{Y}:= \mathbb{Y}(X_{d+1}) := (X_{d+1} - d_2) \cap [N].
\]

Returning to the proof of the lemma,
it suffices to bound $Q_{{\bm \ell}, k}$. Turning to this task, 
we assume without loss of generality
that $|(\Delta_N)_{i,j}| \le 1$. This implies that
\begin{equation}\label{eq:det-Delta}
\left|\det (\Delta_N[X;Y]) \right| \le N^{-\gamma_\star k} k!,
\end{equation}
for every $X, Y \subset [N]$ such that $|X|=|Y|=k$. 
On the other hand, the definition of $d_0=d_0(z,{\bm a})$  and the fact 
that $z \in \cS_\gd$  imply 
that there are no roots of $P_{z, {\bm a}}(\cdot)$ on the unit circle,
hence
we deduce that there exists $\vep_\star=\vep_\star(z,{\bm a}) >0$, 
such that 
\begin{equation}\label{eq:lambda-bd}
 \max\left\{\max_{i=1}^{d_0} |\lambda_i(z)|^{-1} , \max_{i=d_0+1}^d|\lambda_i(z)|\right\} \le 1-\vep_\star.
\end{equation}
Hence, 
\begin{equation}\label{eq:sum-Q-bd}
N^{\gamma_\star |\gd|} \left| \sum_{\bm \ell} Q_{\ell, k}(z)\right| \le k! N^{\gamma_\star (|\gd|-k)}\sum_{\bm \ell} (1-\vep_\star)^{\sum_i \hat \ell_i} |\gL_{{\bm \ell}, k}| \le d! N^{-\gamma_\star}\sum_{\bm \ell} (1-\vep_\star)^{\sum_i \hat \ell_i} |\gL_{{\bm \ell}, k}|,
\end{equation}
where the last inequality follows from the fact that $|\gd| < k \le d$. To finish the proof it remains to find an upper bound on the cardinality of $\gL_{\ell, k}$. We claim that 
\begin{equation}\label{eq:bound-L-ell-k}
|\gL_{{\bm \ell}, k}| \le  \binom{N}{d} \cdot \prod_{i=1}^{d_0} \binom{\hat \ell_i-1}{k+d_2-1} \cdot \prod_{i=d_0+1}^{d} \binom{\hat \ell_i+k+d_2}{k+d_2}.
\end{equation}
Equipped with \eqref{eq:bound-L-ell-k},
it now follows from \eqref{eq:sum-Q-bd} that
\begin{equation}
  \label{eq-102018}
N^{\gamma_\star |\gd|} \left| \sum_{\bm \ell} Q_{\ell, k}(z)\right| =O(N^{d-\gamma_\star}). 
\end{equation}
Since $a_{d_2} \ne 0$ implies that $\{\lambda_\ell(z)\}$ 
are bounded away from zero, \eqref{eq-102018}
together with \eqref{eq:P_k-decompose}  
yield \eqref{eq:rouche-multi-gr-d_0}. 

It remains to establish the bound \eqref{eq:bound-L-ell-k}. To this end, set 
\begin{equation}\label{eq:delta-dfn}
\delta_{i,j}:= \delta_{i,j}(\cX_k):=\left\{ \begin{array}{ll}
x_{i,j} -x_{i+1,j} & \mbox{ for } i \in [d_0] \text{ and } j \in [k+d_2]\\
x_{i+1,1} & \mbox{ for } i \in [d]\setminus [d_0], \ j =1\\
x_{i+1,j}-x_{i,j-1} &\mbox{ for } i \in [d]\setminus [d_0], \ j \in [k+d_2]\setminus \{1\}\\
N+d_2 - x_{i,k+d_2} & \mbox{ for } i \in [d]\setminus [d_0], \ j =k+d_2+1.
\end{array}
\right.
\end{equation}
For the $\{x_{i,j}\}$ to satisfy
$\cX_k \in \gL_{{\bm \ell}, k}$, we observe that the
$\{\delta_{i,j}(\cX_k)\}$'s can be chosen in at most 
\begin{equation}\label{eq:choice-delta-bd}
\prod_{i=1}^{d_0} \binom{\hat \ell_i-1}{k+d_2-1} \cdot \prod_{i=d_0+1}^{d} \binom{\hat \ell_i+k+d_2}{k+d_2}
\end{equation}
ways. Next, recall that $\cX_k \in \gL_{{\bm \ell}, k}$ implies that
\begin{equation}\label{eq:x_d+1-contraint}
x_{1, k+j} = N+j, \quad j=1,2,\ldots, d_2.
\end{equation}
Thus,
$\{x_{1,\ell}\}_{\ell=1}^k$ and $\{\delta_{i,j}(\cX_k)\}$
automatically fix $\cX_k$. Since the number of choices of 
$\{x_{1,\ell}\}_{\ell=1}^k$ is at most $\binom{N}{k} \le \binom{N}{d}$, 
as $k \le d$, for all large $N$, the claim \eqref{eq:bound-L-ell-k} follows
from \eqref{eq:choice-delta-bd}. The proof of the lemma is now complete.
\end{proof}

Next we show that for $z \in S_{\gd}$ the sum $\sum_{k < |\gd|} P_k(z)$ is of smaller order compared to the dominant term $P_{|\gd|}(z)$.

\begin{proof}[Proof of Lemma \ref{lem:rouche-multi-le-d_0}]
We first
claim that for any $k < |\gd|$, the set  $\gL_{{\bm \ell}, k}$
(see \eqref{eq:gL-ell-k}) being nonempty forces
either $\sum_{i=d_0+1}^d \hat \ell_i$ or $\sum_{i=1}^{d_0} \hat \ell_i$ to be 
close to $N$, depending on whether $\gd >0$ or $\gd <0$.
%, then
%the set
%(recall its definition from \eqref{eq:gL-ell-k}) when  
This observation will be then combined with the bounds \eqref{eq:lambda-bd} and \eqref{eq:bound-L-ell-k} to complete the proof. 

%To see the claim on $\gL_{{\bm \ell}, k}$, first let us 
Consider first
the case $\gd =d_1-d_0>0$. For any $k < \gd$ we have that $d_0+k+d_2+1 \le d_1+d_2=d$ and hence for any $\cX_k \in \gL_{{\bm \ell}, k} \subset L_{{\bm \ell}, k}$,
\begin{multline*}
x_{d_0+2,1}+ (x_{d_0+3,2} - x_{d_0+2,1})+\cdots+(x_{d_0+k+d_2+1,k+d_2}- x_{d_0+k+d_2,k+d_2-1})+(N+d_2 - x_{d_0+k+d_2+1,k+d_2})\\
=\sum_{\ell=1}^{k+d_2+1} \delta_{d_0+\ell, \ell} (\cX_k)=  N+d_2.
\end{multline*}
As $\delta_{i,j}(\cX_k) \le \hat \ell_i+k+d_2$ for 
$i \in [d]\setminus [d_0]$ and $j \in [k+d_2+1]$, it further implies that if $\gL_{{\bm \ell}, k} \ne \emptyset$ then we must have 
\begin{equation}\label{eq:hat-ell-lbd-Delta-pos}
\sum_{i= d_0+1}^d \hat \ell_i \ge N+d_2 - (k+d_2) (d-d_0). 
\end{equation}
Next we consider the case $\gd <0$. For any $\cX_k \subset \gL_{{\bm \ell}, k}$ we have that $x_{1,k+1} = N+1$. Therefore
\begin{equation}\label{eq:hat-ell-lbd-pre}
N+1 - x_{d_0+1, k+1} =x_{1,k+1}-x_{d_0+1,k+1}=\sum_{\ell=1}^{d_0} \delta_{\ell, k+1}(\cX_k).
\end{equation}
On other hand, we have that $x_{d+1, \ell} = \ell$ for $\ell \in [d_2]$. Since $x_{i+1,\ell} \le x_{i,\ell} < x_{i+1, \ell+1}$ for any $\ell \in [k+d_2 -1]$, and $\{x_{i, \ell}\}$ are integers, using induction, we further obtain that 
\[
x_{d_0+1, \ell}= \ell, \quad \ell \in [d_2-(d-d_0)],
\]
for any $ \cX_k \in \gL_{{\bm \ell}, k}$. As $ k+1 \le |\gd| = d_0 -d _1 = d_2 - (d-d_0)$ we find that $x_{d_0+1, k+1} =k+1$. Hence, from \eqref{eq:hat-ell-lbd-pre} we deduce that 
\[
\sum_{\ell=1}^{d_0} \delta_{\ell, k+1}(\cX_k) = N-k,
\]
for any $\cX_k \subset \gL_{{\bm \ell}, k}$. Noting that $\delta_{i,k+1}(\cX_k) \le \hat \ell_i -1$ for all $i \in [d_0]$, we obtain
\begin{equation}\label{eq:hat-ell-lbd-Delta-neg}
\sum_{i=1}^{d_0} \hat \ell_i \ge N- (k+d_0).
\end{equation}
Thus, \eqref{eq:hat-ell-lbd-Delta-pos} and \eqref{eq:hat-ell-lbd-Delta-neg} implies that, if $k < |\gd|$ then 
\[
\gL_{{\bm \ell}, k} \ne \emptyset \Longrightarrow \sum_{i=1}^d \hat \ell_i \ge N- d^2.
\]
To complete the proof of the lemma we now use \eqref{eq:P_k-decompose-1}-\eqref{eq:lambda-bd} and \eqref{eq:bound-L-ell-k} to conclude 
that for any $k < |\gd|$,
\[
\left|\sum_{\ell: \gL_{{\bm \ell}, k} \ne \emptyset} Q_{{\bm \ell}, k}(z)\right| \le d! (1-\vep_\star)^{N-d^2}\sum_{i=1}^{N+d_2}|\gL_{{\bm \ell}, k}| \le (1-\bar \vep)^{2N}
\]
for all large $N$, for some sufficiently small $\bar \vep >0$, depending only on $z$ and ${\bm a}$. The proof finishes upon using \eqref{eq:P_k-decompose}.
\end{proof}

\subsection{Lower and upper bounds on the dominant term}\label{sec:dominant}
We will first prove  
Lemma \ref{lem:dom-term}, which is a 
lower bound on the dominant term.
The proof is based on
the following elementary 
anti-concentration bound for homogeneous 
polynomials of independent random variables, which may be of independent interest. 

\begin{proposition}\label{prop:anti-conc}
Fix $k, n \in \N$ and let $\{U_i\}_{i=1}^n$ be a sequence of
independent real-valued random variables,
whose law possesses a
density with respect to the Lebesgue measure which is
uniformly bounded by one. Let $Q_k(U_1,U_2,\ldots,U_n)$ be a 
homogenous polynomial of degree $k$ such that the degree of 
each variable is at most one. That is,
\[
Q_k(U_1,U_2,\ldots,U_n) := \sum_{\cI \in \binom{[n]}{k}} a(\cI) \prod_{i \in \cI} U_{i},
\]
for some collection of complex valued coefficients $\{a(\cI); \, \cI \in \binom{[n]}{k}\}$, where $\binom{[n]}{k}$ denotes the set of all $k$ distinct elements of $[n]$.

Assume that there exists an 
$\cI_0 \in \binom{[n]}{k}$ such that $|a(\cI_0)| \ge c_\star$ for some absolute constant $c_\star>0$. Then for any $\vep \in (0, \corAB{e^{-1}}]$ we have
\[
\P\left( |Q_k(U_1,U_2,\ldots,U_n)| \le \vep\right) \le \corAB{(8e)}^{k} (c_\star \wedge 1)^{-1}\vep \left(\log\left(\frac{1}{\vep}\right)\right)^{k-1}.
\]
\end{proposition}

\begin{proof}
As the densities of $\{U_i\}_{i \in [n]}$ are uniformly bounded by one,
the desired anti-concentration property is immediate for $k=1$. 
To prove the general case, we proceed by induction. 
To this end, we introduce 
some notation. Order the elements of $\cI_0$
and denote them by $i_1^0, i_2^0, \ldots, i_k^0$. For $j \le k$,
define $\cI^0_j:=\{i_j^0, i_{j+1}^0, \ldots, i_k^0\}$. Set 
\[
Q_k^0:=Q_k^0(U_i; i \notin \cI_k^0):=  \sum_{\cI: \cI \supset \cI_k^0} a(\cI) \prod_{\ell \in \cI \setminus \cI_k^0} U_{\ell}  \quad \text{ and } \quad 
Q_k^1:=Q_k^1(U_i; i \notin \cI_k^0):=\quad \sum_{\cI: \cI\cap \cI_k^0= \emptyset} a(\cI) \prod_{\ell \in \cI} U_{\ell}.
\]
For $1 \le j \le k-1$, we iteratively define 
\[
Q_j^0:=Q_j^0(U_i; i \notin \cI_j^0):=  \sum_{\cI: \cI \supset \cI_j^0} a(\cI) \prod_{\ell \in \cI \setminus \cI_j^0} U_{\ell}  \quad \text{ and } \quad 
Q_j^1:=Q_j^1(U_i, i \notin \cI_j^0):=\quad \sum_{\substack{\cI: \cI \supset \cI_{j+1}^0\\  i_j^0\notin \cI }} a(\cI) \prod_{\ell \in \cI \setminus \cI_{j+1}^0}U_{\ell}.
\]
Equipped with the above notations we see that
\[
Q_k(U_1,U_2, \ldots, U_n) =:Q_{k+1}^0 = U_{i^0_k} \cdot Q_k^0 + Q_k^1, \qquad 
Q_{j+1}^0 = U_{i_{j}^0} \cdot Q_{j}^0 + Q_{j}^1, \ j=1,2,\ldots, k-1,
\]
and
\(
Q_1^0= a(\cI_0) 
\). We will prove inductively that
%claim that 
\begin{equation}\label{eq:anti-conc-induction}
\P\left( |Q_j^0| \le \vep \right) \le \corAB{(8e)}^{j-1} (c_\star \wedge 1)^{-1}\vep \left(\log\left(\frac{1}{\vep}\right)\right)^{j-2}, \quad j=2,3,\ldots, k+1,
\end{equation}
from which the desired anti-concentration bound follows by taking $j=k+1$. 
Hence, it only remains to prove \eqref{eq:anti-conc-induction}. 

For $j=2$, $Q_j^0$ is a homogeneous
polynomial of degree $1$ in the variables $U_i$, and
\eqref{eq:anti-conc-induction}  follows from the assumptions 
on $\{U_\ell\}_{\ell=1}^n$ and the fact that $|a(\cI_0)| \ge c_\star$. 
Assuming that \eqref{eq:anti-conc-induction} holds for $j=j_*$ and fixing $\delta \in (0,1)$, 
we have that with  $C_j:= \corAB{(8e)}^{j-1} (c_\star \wedge 1)^{-1}$,
\begin{align}\label{eq:anti-conc-split}
\P\left( \left|Q_{j_*+1}^0 \right| \le \vep\right)  & \le \P \left( \left| Q_{j_*}^0\right| \le \delta \right) + \E\left[ \P\left(  \left| U_{i_{j_*}^0} + \frac{Q_{j_*}^1 }{Q_{j_*}^0 }\right| \le \frac{\vep}{|Q_{j_*}^0 |} \bigg| \, U_i, i \notin \cI_{j_*}^0\right)\cdot {\bm 1} \left( \left|Q_{j_*}^0\right| \ge \delta\right) \right] \notag\\
& \le C_{j_*} \delta \log \left( \frac{1}{\delta}\right)^{j_*-2} + 2{\vep} \cdot \E \left[ |Q_{j_*}^0|^{-1} {\bm 1} \left( \left|Q_{j_*}^0\right| \ge \delta\right) \right],
\end{align}
where we have used the fact that $Q_{j_*}^1$ and $Q_{j_*}^0$ are independent of $U_{i_{j_*}}^0$, and the bound on the density for the latter.
Using integration by parts, for any probability measure $\mu$ supported on $[0,\infty)$ we have that 
\[
\int_\delta^{\corAB{e^{-1}}} x^{-1}  d\mu(x) = \corAB{e} \mu([\delta, 1])  + \int_\delta^{\corAB{e^{-1}}} \frac{\mu([\delta,t])}{t^2} dt.
\]
Therefore, using the induction hypothesis, 
\begin{align*}
\E \left[ |Q_{j_*}^0|^{-1} {\bm 1} \left( \left|Q_{j_*}^0\right| \ge \delta\right) \right] \le \corAB{e}+ \E \left[ |Q_{j_*}^0|^{-1} {\bm 1} \left( \left|Q_{j_*}^0\right| \in [\delta,\corAB{e^{-1}}]\right) \right] & \le 2\corAB{e} + \int_\delta^{\corAB{e^{-1}}} \frac{ \P(|Q_{j_*}^0| \le t)}{t^2} dt \\
& \le 2 \corAB{e}+  C_{j_*} \int_\delta^{\corAB{e^{-1}}} t^{-1} \left(\log\left(\frac{1}{t}\right)\right)^{j_*-2} dt \\
& \le 2 \corAB{e}+ \frac{C_{j_*}}{j_*-1} \left(\log\left(\frac{1}{\delta}\right)\right)^{j_*-1}.
\end{align*}
\corAB{Since for $\delta \le e^{-1}$ we have that $\log(1/\delta) \ge 1$}, combining the above with \eqref{eq:anti-conc-split} and setting $\delta=\vep$ we establish \eqref{eq:anti-conc-induction} for $j=j_*+1$. This completes the proof.
\end{proof}

Equipped with Proposition \ref{prop:anti-conc} we now begin the proof of Lemma \ref{lem:dom-term}.

\begin{proof}[Proof of Lemma \ref{lem:dom-term}] 
Recalling \eqref{eq:P_k-z} we note that $P_{|\gd|}(z)$ is a homogeneous polynomial of degree $|\gd|$ in the entries of the noise matrix $\Delta_N$ such that the degree of each entry of $\Delta_N$ is one. Therefore, to apply Proposition \ref{prop:anti-conc} we only need to show that there exists  $X, Y \subset [N]$ with $|X|=|Y|=|\gd|$ such $\det(T_N(z) [X^c;Y^c])$ is bounded below. The choice of such subsets will depend on the sign of $\gd$. Hence, the proof is split into two cases.

Considering the case $\gd >0$ we set $X=[N]\setminus [N-\gd]$ and $Y =[\gd] $. Recalling Definition \ref{dfn:toep-shifted} we find that 
\[
T_N(z)[X^c;Y^c]= T_{N-\gd}(z; d_1-\gd, d_2+\gd).
\]
We apply Widom's result on the determinant of finitely banded Toeplitz matrices, in particular \cite[Theorem 2.8]{bottcher-finite-band} to deduce that for any $z \in \C\setminus \cN$, one has
\[
\det (T_N(z)[X^c; Y^c]) = \sum_{\cI \in \binom{[d]}{d_1-\gd}} C_{\cI} \cdot a_{d_1}^N \prod_{\ell \in \cI} \lambda_\ell(z)^N,
\]
for some collection of coefficients $\{C_\cI\}$, where recall that $\cN$ is the collections of $z$'s such that $P_{z,{\bm a}}(\cdot)$ has double roots. 
Furthermore, the coefficients $\{C_\cI\}$ are
bounded both below and above, for any $z \in B_\C(0,R)\setminus \cN$. As $z \in \cS_\gd$ and $d_1-\gd =d_0(z)$, using \eqref{eq:lambda-bd} we therefore deduce that there exists some small positive constant $c_0>0$ so that, for all large $N$,
\begin{equation}\label{eq:det-lbd}
|\det (T_N(z)[X^c; Y^c]) | \ge c_0 \cdot |a_{d_1}|^N \prod_{\ell=1}^{d_0(z)} |\lambda_\ell(z)|^N.
\end{equation}
From the definition of $\Delta_N$ it follows that for the above choices of $X$ and $Y$ the determinant of $\Delta_N[X;Y]$, ignoring the factor $N^{-\gamma_\star \gd}$, is a homogeneous polynomial of degree $\gd$ of independent uniformly bounded random variables with uniformly bounded densities. Therefore, we are in a position to apply Proposition \ref{prop:anti-conc}. 

Without loss of generality, 
assuming that the densities of $\{(\Delta_N)_{i,j}\}_{i,j=1}^N$ 
are uniformly bounded by one, we apply Proposition \ref{prop:anti-conc} for 
\[
\frac{N^{\gamma_0 \gd}\cdot P_{\gd}(z)}{|a_{d-1}|^N\prod_{\ell=1}^{d_0(z)} |\lambda_\ell(z)|^N} 
\]
with $c_\star = c_0$ and $\vep = N^{-\vep_0/2} c_\star$ to arrive at \eqref{eq:dom-term} for any $z \in \cS_\gd \setminus \cN$. As $\cN$ contains at most finitely many points this proves the lemma when $\gd >0$.

Turning to prove the same for $\gd <0$, we reverse the roles of $X$ and $Y$. That is, we now set $X= [-\gd]$ and $Y= [N] \setminus [N+\gd]$ and follow the same steps as above. 

For $\gd=0$ the proof is straightforward. From  its definition we have $P_0(z)= \det(T_N(z))$. Upon setting $X=Y=\emptyset$ in \eqref{eq:det-lbd} the result is immediate. Now the proof of the lemma is complete.
\end{proof}

We end this section with the proof of Lemma \ref{lem:dom-term-ubd}. Its proof is very similar to that of Lemma \ref{lem:rouche-multi-gr-d_0}. Hence, only an outline is provided.

\begin{proof}[Proof of Lemma \ref{lem:dom-term-ubd}]
We split the proof into two cases: $\gd \ne 0$ and $\gd =0$. First, let us consider $\gd \ne 0$. As $\gamma _\star > d$ we find from \eqref{eq:P_k-decompose}-\eqref{eq:lambda-bd} and \eqref{eq:bound-L-ell-k} that
\[
\frac{\left|P_{|\gd|}(z)\right|}{|a_{d_1}|^{N-k} \cdot \prod_{i=1}^{d_0}|\lambda_i(z)|^{N+d_2}} = O(N^{d-\gamma_\star |\gd|}) = O(N^{d -\gamma_\star})=o(1). 
\]
If $\gd =0$ then the desired result follows from Widom's result
(see \cite[Theorem 2.8]{bottcher-finite-band}). 
\end{proof}

\section{Proof of Theorem \ref{thm:GWZ}}\label{sec:proof-GWZ}
We recall from Section \ref{sec:proof-outline} that to prove Theorem \ref{thm:GWZ} it suffices to establish \eqref{eq:log-pot-diff}. 
As outlined there, the key to the latter 
is to bound the 
difference of the mass
of intervals near zero under the measures $\nu^z_{A_N+\Delta_N}$ and $\nu^z_{A_N+E_N}$, the empirical distribution of the singular values of $A_N(z)+\Delta_N$ and $A_N(z)+N^{-\gamma} E_N$, respectively, where $A_N(z):=A_N -z \Id_N$. 
This in turns will be achieved by controlling the differences of
the Stieltjes transforms of the corresponding measures. 
So, we begin this section with its definition. 
\begin{dfn}
The \textit{Stieltjes transform} of
a probability measure $\mu$ on $\R$
 is defined as
\[
G_\mu(\xi):=\int \frac{1}{\xi-y}\, \mu(dy), \, \xi \in \C\setminus \R.
\]
\end{dfn}
To obtain a bound on the probability of any interval under $\mu$ from that of $G_\mu(\cdot)$ we use the following two inequalities. These are a consequence of \cite[Eqns.~(6)-(8)]{GWZ}: for any $\tau, \varrho >0$, and $a, b \in \R$ such that
$b - a > \varrho$ we have
\begin{equation}\label{eq:pr-stielt-ubd}
\mu([a,b]) \le \int_{a-\varrho}^{b+\varrho} |\Im G_\mu(x + \mathrm{i} \tau)| dx + \frac{\tau}{\varrho},
\end{equation}
and
\begin{equation}\label{eq:pr-stielt-lbd}
\mu([a,b]) \ge \int_{a+\varrho}^{b-\varrho} |\Im G_\mu(x + \mathrm{i} \tau)| dx  - \frac{\tau}{\varrho}.
\end{equation}
Now to find a difference of the Stieltjes transforms of $\nu^z_{A_N+\Delta_N}$ and $\nu^z_{A_N+E_N}$ we also need the symmetrized form of the Stieltjes transform, as follows.
For a $N \times N$ matrix $C_N$, define
\begin{equation}\label{eq:tilde-matrix}
\widetilde{C}_N:= \begin{bmatrix} 0 & C_N \\ C_N^* & 0 \end{bmatrix}
\end{equation}
and the Stieltjes transform
\[
  G_{C_N}(\xi):= \frac{1}{2N}\tr \left(\xi- \widetilde{C}_N \right)^{-1}, \, \xi \in \C \setminus \R.
\]
$G_{C_N}(\cdot)$ is the Stieltjes transform
of the symmetrized version of the empirical measure of the singular values of $C_N$. Equipped with the above notation we have the following lemma.
\begin{lemma}
For $C_N$ and $D_N$ any $N \times N$ matrices,
\begin{equation}\label{eq:stieltjes-diff}
  |G_{C_N}(\xi) - G_{D_N}(\xi)| \le \frac{1}{\sqrt{N}}\cdot \frac{\|C_N -D_N\|_{{\rm HS}}}{(\Im (\xi))^2},
\end{equation}
where $\| \cdot \|_{{\rm HS}}$ denotes the Hilbert-Schmidt norm.
\end{lemma}
\begin{proof}
Using the resolvent identity we have that
\begin{equation}
  \label{eq-201018b}
G_{C_N}(\xi) - G_{D_N}(\xi) = \frac{1}{2N}\tr \left[ \left(\xi - \widetilde D_N\right)^{-1} \cdot \left(\xi - \widetilde C_N\right)^{-1} \cdot (\widetilde C_N - \widetilde D_N)\right].
\end{equation}
Recall 
the following version of the 
Cauchy-Schwarz inequality: for any two $(2N) \times (2N)$ matrices $A_N$ and $B_N$
\begin{equation}
  \label{eq-201018a}
\frac{1}{2N}|\tr (A_NB_N)| \le \frac{1}{2N}\sqrt{\tr (A_N^*A_N)} \cdot \sqrt{\tr (B_N^* B_N)} \le \|A_N\| \cdot \frac{1}{\sqrt{2N}}\|B_N\|_{{\rm HS}}.
\end{equation}
Since 
for any Hermitian matrix $H_N$ one has $\|(\xi - H_N)^{-1}\| \le 1/|\Im (\xi)|$,
the claim follows from \eqref{eq-201018b} upon using
\eqref{eq-201018a} with $A_N= \left(\xi - \widetilde D_N\right)^{-1} \cdot \left(\xi - \widetilde C_N\right)^{-1}$ and
$B_N=\widetilde C_N-\widetilde D_N$.
\end{proof}

As a last preliminary step, we need the following easy lemma.
\begin{lemma}
  \label{lem-(b)}
  For any probability measure $\mu$, 
  %\item[(b)] For Lebesgue a.e.~$z \in B_\C(0,R_0/2)$,
\begin{equation}\label{eq:regularity}
\lim_{\vep \downarrow 0} \int \log |x-z| {\bf 1}_{\{|x-z| \le \vep\}} d\mu(x) =0
\end{equation}
for \corA{Lebesgue} almost every $z\in \C$.
\end{lemma}
\begin{proof} 
  For $\vep<1$, set
  $F(z,\vep):=\int \log(1/ |x-z|) {\bf 1}_{\{|x-z| \le \vep\}} d\mu(x)$. 
  Fix $z_0\in C$. Note that, by Fubini's theorem,
  $$ \int_{B_\C(z_0,1)} F(z,\vep)dz< \pi\vep^2\log (1/\vep)
  \to_{\vep\to 0} 0.$$
  In particular, for any $\delta>0$,
  with ${\mathcal A}_\vep(\delta):=\{z\in B_\C(z_0,1): F(z,\vep)>\delta\}$,
  we obtain that 
  $$\mbox{\rm Leb}({\mathcal A}_\vep(\delta))\to_{\vep \to 0}
  0.$$
  In particular, $\mbox{\rm Leb}(\cap_{\vep<1}{\mathcal A}_\vep(\delta))=0$.
  Using the monotonicity of $F(z,\vep)$, we conclude that for
  Lebesgue almost every $ z\in B_\C(z_0,1)$,
  $\limsup_{\vep\to 0} F(z,\vep)\leq \delta$. Taking a sequence $\delta_n\to 0$
  gives \eqref{eq:regularity}, first for Lebesgue almost every $ z\in B_\C(z_0,1)$, and then for almost every $z$.
\end{proof}

We are now ready to prove Theorem \ref{thm:GWZ}.

\begin{proof}[Proof of Theorem \ref{thm:GWZ}]
To establish \eqref{eq:conv-log-pot} we first claim
that $\nu^z_{A_N+N^{-\gamma} E_N} \Rightarrow \mu_z$, in probability, 
for Lebesgue a.e.~$z \in B_\C(0,R_0)$, 
where $\mu_z$ is the law of $|X-z|$ and $X \sim \mu$. 
The argument is similar to that employed in the proof of Theorem
\ref{thm:main}. 
Write
$B_N:= A_N+ N^{-\gamma} E_N$ and
$B_N(z):=B_N -z \Id_N$.
We have that $\nu^z_{A_N}\Rightarrow \mu_z$ by assumption (b), 
while
% $\gamma >\frac12$, 
Assumption \ref{ass:noise-matrix}(i) and Markov's inequality imply that 
\begin{equation}\label{eq:HS-Markov}
\P\left( \|E_N\|_{{\rm HS}} \ge N^{1+\frac{\gamma -\frac12}{2}} \right) \le N^{-(\gamma -\frac12)} \cdot \frac{\E \|E_N\|_{{\rm HS}}^2}{N^2}= o(1), 
\end{equation}
for any $\gamma>1/2$. 
On the other hand, by the Hoffman-Wielandt inequality, see
\cite[Lemma 2.1.19]{AGZ}, the map
$D_N\mapsto L_N^D$, viewed as a map from the space of $N\times N$
Hermitian matrices 
equipped with the normalized
Hilbert-Schmidt norm $N^{-1/2} \|\cdot\|_{\rm HS}$ 
to the space of probability measures
equipped with the weak topology, is continuous. 
Note that for any matrix,
the singular values of $A$ are the same as the modulus of the eigenvalues
of the matrix
\[\left(\begin{array}{ll}
    0&A\\A^*&0
\end{array}\right),\]
up to double the multiplicity for each singular value.
In particular,
%$A_N(z)=A_N-zI_N$ has uniformly bounded operator norm,
%only finitely many nonzero elements in each row and
%they are all uniformly bounded, 
%we have that 
\[N^{-1/2}\|(B_N(z)B_N(z)^*)^{1/2}
  -(A_N(z)A_N(z)^*)^{1/2}\|_{\rm HS}
\leq  CN^{-(\gamma+1/2)}\|E_N\|_{\rm HS} 
%+
%N^{-(2\gamma+1/2)}(\|E_N\|_{\rm HS})^2
\to_{N\to \infty} 0,\]
in probability, by \eqref{eq:HS-Markov}.
%Since the singular values of $B_N(z)$ are the square roots
%of the eigenvalues of $B_N(z)B_N(z)^*$, and
%the
%map $x\mapsto \sqrt{x}$ is continuous on $\R$, 
We conclude from that and the above mentioned continuity of the 
empirical measure in  the (normalized) Hilbert-Schmidt norm 
that 
\begin{equation}
  \label{eq-laundry}
  \nu^z_{B_N}\Rightarrow \mu_z, \quad \mbox{\rm  in probability},
\end{equation}
as claimed.

To complete the proof we need to extend the convergence  of
$\nu^z_{B_N}$\
%eqref{eq:sing-dist-conv} 
%to yield 
to the convergence of the integral of $\log(\cdot)$ against this measure. 
To this end, using \eqref{eq:HS-Markov} again and the fact that
the operator norm of $A_N(z)$ is bounded,
we see that there exists a compact set $\mathbb{K} \subset\R$ such that for any $z \in B_\C(0,R_0)$
\[
\P(\nu^z_{B_N}(\mathbb{K}^c) >0) =o(1).
\]
Hence, for any $\vep >0$,
\[
\int_\vep^\infty \log(x) d\nu_{B_N}^z(x) \to \int_\vep^\infty \log(x) d\mu_z(x) = \int \log(|x-z|) {\bf 1}_{\{|x-z| > \vep\}}d\mu(x), \qquad \text{ in probability},
\]
for Lebesgue a.e.~$z \in B_\C(0,R_0)$. Note  that
\[
\mathcal{L}_{L_N^{B}}(z) = \frac{1}{N}\log |\det(B_N(z))| = \int \log(x) d\nu_{B_N}^z(x).
\]
Thus, \eqref{eq:regularity}
together with \eqref{eq-laundry} imply that it only remains to show that given any $\delta> 0$, there exists $\vep_0(\delta)$ such that for any $\vep \le \vep_0(\delta)$
\begin{equation}\label{eq:log-near-0-negligible-1}
\limsup_{N \to \infty} \P\left( \left| \int_0^\vep \log(x) d\nu_{B_N}^z(x)\right| \ge C_0\delta\right) =0,
\end{equation}
for Lebesgue a.e.~$z \in B_\C(0,R_0)$ and some large constant $C_0$. To prove 
this, we
 first show that an analogue of \eqref{eq:log-near-0-negligible-1} holds for the empirical measure of the singular values of $A_N+\Delta_N$. 

Turning to do this task, using \eqref{eq:E-op-norm} and 
arguing similarly to  the steps leading to \eqref{eq-laundry},
we obtain that 
\[
 \nu_{A_N+\Delta_N}^z \Rightarrow \mu_z, \qquad \text{ in probability},
\]
for Lebesgue a.e.~$z \in B_\C(0,R_0)$, and further,
%. Hence, same lines of argument as above further implies that, 
for any $\vep >0$,
\begin{equation}\label{eq:sing-val-conv}
\int_\vep^\infty \log (x) d\nu_{A_N+\Delta_N}^z(x) \to \int_\vep^\infty \log|x-z| d\mu_z(x), \qquad \text{ in probability},
\end{equation}
for Lebesgue a.e.~$z \in B_\C(0,R_0)$. 
Together with assumptions (b)-(c), we conclude that
for Lebesgue a.e.~$z \in B_\C(0,R_0)$, given any $\delta >0$, 
there exists $\vep_0(\delta)$ such that for all $\vep \le 2 \vep_0(\delta)$,
\begin{equation}\label{eq:log-near-0-negligible}
\limsup_{N \to \infty} \P\left( \left| \int_0^\vep \log(x) d\nu_{A_N+\Delta_N}^z(x)\right| \ge \delta\right) =0.
\end{equation}
Having shown \eqref{eq:log-near-0-negligible} we now proceed to the proof of \eqref{eq:log-near-0-negligible-1}. Using Assumption \ref{ass:noise-matrix}(ii) we see that there exists a sufficiently large constant $\kappa_\star$ such that
\[
\P(s_{\min}(B_N(z)) \le N^{-\kappa_\star}) = \P(s_{\min}(N^\gamma A_N(z)+E_N) \le N^{\gamma - \kappa_\star}) =o(1),
\]
where $s_{\min}(H)$ is the minimal singular value of a matrix $H$.
Hence, 
\begin{equation}\label{eq:log-near-0-negligible-2}
\limsup_{N \to \infty} \P\left( \left|\int_0^{N^{-\kappa_\star}} \log (x) d\nu_{B_N}^z(x)\right| \ge \delta\right) =0.
\end{equation}
Now to control the integral of $\log(\cdot)$ over $(N^{-\kappa_\star}, \vep)$ we apply \eqref{eq:stieltjes-diff} to deduce that 
\begin{equation}\label{eq:stieltjes-diff-1}
|G_{A_N(z)+\Delta_N}(x+\mathrm{i} \tau) - G_{B_N(z)}(x+\mathrm{i} \tau)| \le \frac{N^{-\gamma}\|E_N\|_{{\rm HS}}+\|\Delta_N\|_{{\rm HS}}}{\sqrt{N}\tau^2} \le N^{-\delta'/2},
\end{equation}
on the event 
\[
\Omega_N:= \left\{\|E_N\|_{{\rm HS}} \le N^{1+\frac{\gamma -\frac12}{2}}\right\} \cap \{\|\Delta_N\| \le N^{-\gamma_0}\},
\]
where $\tau=N^{-\delta'/4}$ and $\delta'=\min\{\frac12(\gamma - \frac12), \gamma_0\}$. 

Let $\widetilde \nu_{A_N+\Delta_N}^z$ and $\widetilde \nu_{B_N}^z$ denote
the symmetrized versions of the probability measures $ \nu_{A_N+\Delta_N}^z$ and $ \nu_{B_N}^z$, respectively. 
%Denote, $\nu_{\model_N}^z$ and $\nu_{\hat\model_N}^z$ to be the empirical measure of the singular values of $\model_N - z\Id$ and $\hat\model_N- z\Id$, respectively.
Setting $\varrho = N^{-\delta'/8}$, $\kappa=\delta'/16$, and using \eqref{eq:pr-stielt-ubd}-\eqref{eq:pr-stielt-lbd} \corA{and \eqref{eq:stieltjes-diff-1}} in the second inequality,
we have that
\begin{align}
-\int_{N^{-\kappa_\star}}^{N^{-\kappa}} \log (x) d\tilde\nu_{B_N}^z(x) & \le \kappa_\star \log N \cdot \nu_{B_N}^z([N^{-\kappa_\star}, N^{-\kappa}])\notag\\
 & \le \kappa_\star \log N \left(2 N^{- \delta' /8} \corA{+N^{-\delta'/2}} +2 \tilde \nu_{A_N+\Delta_N}^z([N^{-\kappa_\star}- 2{\varrho}, N^{-\kappa}+2{\varrho}]) \right)\notag\\
 & \le  \kappa_\star \log N \left(\corA{3} N^{- \delta' /8} +
 4\tilde \nu_{A_N+\Delta_N}^z([0, 2N^{-\kappa}]) \right) \notag\\
 & \le \corA{3} \kappa_\star \log N \cdot N^{-\delta' /8}-\frac{8\kappa_\star}{\kappa} \int_0^{2N^{-\kappa}} \log (x) d\tilde
 \nu_{A_N+\Delta_N}^z(x), \label{eq:log-near-0-1}
\end{align}
on $\Omega_N$, for all large $N$, where in the third inequality we used the symmetry of $\tilde \nu_{A_N+\Delta_N}^z$
%second last step we have used the fact that $\nu_{\hat\model_N}^z$ is supported on $[0,\infty)$
and $\varrho= o(N^{-\kappa})$.

It remains to bound the integral of $\log(\cdot)$ over $(N^{-\kappa}, \vep)$.
Toward this,
using integration by parts we note that, for $0 \le a_1 < a_2  <1$ and any probability measure $\mu$ on $\R$,
\begin{equation}\label{eq:log-by-parts}
-\int_{a_1}^{a_2} \log (x) d\mu(x) = -\log(a_2) \mu([a_1,a_2]) + \int_{a_1}^{a_2} \frac{\mu([a_1,t])}{t} dt.
\end{equation}
Arguing as in \eqref{eq:log-near-0-1} we obtain
\begin{align}
\int_{N^{-\kappa}}^\vep \frac{\tilde
  \nu_{B_N}^z([N^{-\kappa}, t])}{t} dt& \le \corA{3} N^{- \delta' /8}
  \int_{N^{-\kappa}}^\vep \frac{1}{t} dt +
  \int_{N^{-\kappa}}^\vep \frac{\tilde \nu_{A_N+\Delta_N}^z([N^{-\kappa}/2, t+N^{-\kappa}])}{t} dt \notag\\
& \le \corA{3 } \kappa N^{-\delta' /8} \cdot \log N + \int_{N^{-\kappa}/2}^{2\vep} \frac{\tilde \nu_{A_N+\Delta_N}^z([N^{-\kappa}/2, t])}{t} dt,  \label{eq:log-near-0-2}
\end{align}
where in the last step we have used the fact that $t + N^{-\kappa} \le 2t$ for any $t \ge N^{-\kappa}$, and a change of variables.  Similar reasoning yields that
\begin{equation}
-\log(\vep) \tilde \nu_{B_N}^z([N^{-\kappa},\vep]) \le - \log(\vep) \left( {\corA{3 }  N^{-\delta'/8}} +
\tilde \nu_{A_N+\Delta_N}^z([N^{-\kappa}/2,2\vep])\right). \label{eq:log-near-0-3}
\end{equation}
Thus combining \eqref{eq:log-near-0-1}, and \eqref{eq:log-near-0-2}-\eqref{eq:log-near-0-3}, and using \eqref{eq:log-by-parts} we deduce that for $\vep \le\vep_0(\delta)$ sufficiently small and all large $N$,
\[
-\int_{N^{-\kappa_\star}}^{\vep} \log (x) d\tilde
\nu_{B_N}^z(x)  \le C_0' \left[ \log N \cdot N^{- \delta'/8}-\int_0^{2\vep}\log(x) d\tilde \nu_{A_N+\Delta_N}^z(dx)\right],
\]
on the event $\Omega_N$, where $C_0'$ is some large constant. 
Finally, \eqref{eq:HS-Markov} and \eqref{eq:E-op-norm} imply that $\P(\Omega_N^c)=o(1)$. Therefore, combining
\eqref{eq:log-near-0-negligible} and \eqref{eq:log-near-0-negligible-2}, the claim in \eqref{eq:log-near-0-negligible-1} now follows. This completes the proof of the theorem.
\end{proof}

\appendix
\section{Some algebraic facts}\label{sec:alg-results}
{In this section we collect a couple of standard matrix results which have been used in the proofs appearing in Section \ref{sec:proof-existence-of-Delta}.}

{The first result shows that the determinant of the sum of the two matrices can be expressed as a linear combination of products of the determinants of appropriate sub-matrices. The proof follows from the definition of the determinant,
  see e.g. \cite{college}.} {We adopt the convention that} the determinant of the matrix of size zero is one. For an $N\times N$ matrix $A$, and $\row,\col\subseteq [N]$ we write $A[X;Y]$ for the sub-matrix of $A$ which consists of the rows in $\row$ and the columns in $\col$.
\begin{lemma}\label{lem:cauchy-binet}
{ For any $N\times N$ matrices $A$ and $B$ we have}
\begin{equation}\label{eq:det_decomposition}
        \det(A+B) = \sum_{\substack{\row,\col \subset [N] \\ |\row|=|\col|}} (-1)^{\sgn(\sigma_\row)\sgn(\sigma_\col)} \det(A[{\row}^c; {\col}^c])\det(B[\row; \col]),
\end{equation}
where ${\row}^c:=[N]\setminus \row$, ${\col}^c:=[N] \setminus \col$ and $\sigma_Z$ for $Z\in\{\row,\col\}$ is the permutation on $[N]$ which places all the elements of $Z$ before all the elements of ${Z}^c$, but preserves the order of elements within the two sets.
\end{lemma}
The next lemma evaluates 
the determinant of any sub-matrix of a bidiagonal matrix.

\begin{lemma}[{\cite[Lemma 2.2]{FPZ}}]\label{lem:bidiagonal-det}
Let $A_N$ be an upper bi-diagonal matrix and $X, Y \subset [N]$ such that $|X|=|Y|$. Then $\det(A_N[X;Y])$ equals the product of the diagonal entries of $A_N[X;Y]$. 
\end{lemma}

The next lemma, which follows readily from Lemma 
\ref{lem:bidiagonal-det},
evaluates
the determinant of any sub-matrix of a bidiagonal Toeplitz matrix.
\begin{lemma}[{\cite[Lemma 2.3]{FPZ}}]\label{lem:bidiagonal-det-1}
Let $A_N = J_N+ \mathfrak{z} \Id_N$, $\mathfrak{z} \in \C$, $X =\{x_1 <x_2<\cdots <x_k\} \subset [N]$, and $Y=\{y_1<y_2<\ldots<y_k\}\subset [N]$. Then, with
$y_{k+1}=\infty$,
\[
\det(A_N[{X}^c; {Y}^c]) = \mathfrak{z}^{y_1-1} \cdot \left(\prod_{i=2}^k  \mathfrak{z}^{y_i-x_{i-1}-1}\right) \cdot \mathfrak{z}^{N- x_k} {\bf 1} \left\{  y_i \le x_i < y_{i+1}, i \in [k] \right\}.
\]
\end{lemma}

%For $A_N = \mathfrak{z} J_N + \Id_N$ the determinant of $A_N[X^c; Y^c]$ is evaluated in \cite[Lemma 2.3]{FPZ}. Lemma \ref{lem:bidiagonal-det-1} follows from there. 

%Here we provide a short outline.

%\begin{proof}[Proof of Lemma \ref{lem:bidiagonal-det-1}]
%Writing $\COMP{X}=\{w_1 <w_2< \cdots <w_{N-k}\}$, $\COMP{Y}=\{z_1 < z_2< \cdots <z_{N-k}\}$ and using Lemma \ref{lem:bidiagonal-det} we find that 
%\begin{equation}\label{eq:det-simplify}
%\det(A_n[\COMP{X}; \COMP{Y}] )= \prod_{\ell=1}^{N-k} (A_N)_{w_\ell, z_\ell}.
%\end{equation}
%Using the upper bi-diagonal structure of $A_n$ it can be deduced that the \abbr{RHS} of \eqref{eq:det-simplify} equals zero unless
%\[
%y_1 \le x_1 \lneq y_2  \le x_2 \lneq \cdots \lneq y_k \le x_k.
%\]
%Now the rest is immediate and further details are omitted.
%\end{proof}

%\subsection*{\red{Comments}} Add difficulties in treating infinite symbols. Slowly growing bandwidth. Study of outliers. Connection to pseduospectra. Originality of the methods. Not a culmination of works \red{[2,6,7]}.

\end{document}